\documentclass[11pt,reqno]{amsart}
\usepackage[T1]{fontenc}
\usepackage[utf8]{inputenc}
\usepackage{lmodern,microtype}
\usepackage[a4paper,textwidth=158mm,textheight=235mm,centering,headheight=12pt]{geometry}
\usepackage{amsmath,amssymb,amsthm,mathtools}
\usepackage{booktabs,longtable,array,tabularx,enumitem}
\usepackage{tikz,placeins}
\usetikzlibrary{calc,positioning}
\usepackage{hyperref}
\hypersetup{hidelinks,pdftitle={Orbit stratifications and smooth topology of positively curved six-manifolds with rank-one nonabelian symmetry},pdfauthor={},pdfsubject={Transformation groups, positive sectional curvature, smooth six-manifolds}}
\numberwithin{equation}{section}
\newtheorem{theorem}{Theorem}[section]
\newtheorem{proposition}[theorem]{Proposition}
\newtheorem{lemma}[theorem]{Lemma}
\newtheorem{corollary}[theorem]{Corollary}
\theoremstyle{definition}

\theoremstyle{remark}

\newcommand{\SO}{\mathrm{SO}}
\newcommand{\SU}{\mathrm{SU}}
\newcommand{\OO}{\mathrm O}

\newcommand{\Sp}{\mathrm{Sp}}
\newcommand{\CP}{\mathbb{CP}}
\newcommand{\RP}{\mathbb{RP}}

\newcommand{\Z}{\mathbb Z}
\newcommand{\Q}{\mathbb Q}
\newcommand{\R}{\mathbb R}
\newcommand{\C}{\mathbb C}
\newcommand{\HH}{\mathbb H}
\newcommand{\Wflag}{F_{\mathrm W}}
\newcommand{\Eflag}{F_{\mathrm E}}
\newcommand{\id}{\mathrm{id}}
\newcommand{\eps}{\varepsilon}
\newcommand{\Ad}{\operatorname{Ad}}
\newcommand{\diag}{\operatorname{diag}}
\newcommand{\PD}{\operatorname{PD}}

\newcommand{\Hess}{\operatorname{Hess}}

\newcommand{\Sec}{\operatorname{sec}}

\newcommand{\ang}[1]{\langle #1\rangle}

\newcommand{\one}{\mathbf 1}

\newcommand{\arxiv}[1]{\href{https://arxiv.org/abs/#1}{arXiv:#1}}
\newcommand{\doi}[1]{\href{https://doi.org/#1}{doi:\nolinkurl{#1}}}
\setlist[enumerate]{label=\textup{(\roman*)},leftmargin=2em,itemsep=3pt,topsep=4pt}
\tikzset{sv/.style={circle,fill=black,draw=black,inner sep=2.0pt},ev/.style={circle,fill=white,draw=black,line width=.7pt,inner sep=2.3pt},edge/.style={draw=black,line width=.65pt},lab/.style={fill=white,inner sep=1pt,font=\small}}

\title[Six-manifolds with rank-one nonabelian symmetry]{Orbit stratifications and smooth topology of positively curved six-manifolds with rank-one nonabelian symmetry}
\author{Minbo Gao$^1$}
\author{Yuhang Liu$^{*2}$}
\email{yuhang.liu02@xjtlu.edu.cn}
\date{September 29, 2026}

\begin{document}
\footnotetext[1]{Institute of Software, Chinese Academy of Sciences}
\footnotetext[2]{Department of Applied Mathematics, Xi'an Jiaotong-Liverpool University}
\begin{abstract}
We study closed simply connected six-manifolds with positive sectional curvature and an effective isometric action of $\SO(3)$ or $\SU(2)$. Starting from Liu's corrected orbit-space analysis, we refine the stratification and recover smooth topology in several previously unresolved cases. For a boundaryless $\SO(3)$ quotient, singular isotropy is connected, the weighted nonprincipal graph belongs to nine necessary families, and every action with only cyclic exceptional isotropy is classified equivariantly. The remaining branching families admit explicit rational intersection and Pontryagin formulas. For a three-ball $\SO(3)$ quotient, the uniform-boundary configurations give sphere bundles over $S^4$ and a completely distinguished odd-integer family of non-spin six-manifolds. A doubled-quotient argument excludes the configuration with two boundary fixed points and an interior singular orbit, while the configuration with one fixed point and an exceptional interval has underlying manifold $\CP^3$. For a boundaryless $\SU(2)$ quotient, an adjoint-weight matching argument reduces the exceptional three-orbit case to a unique action on the Eschenburg flag and identifies every two-orbit candidate with a weighted action on $S^2\times S^4$. These classifications imply rational ellipticity throughout the stated symmetry class. We also establish restricted curvature obstructions for the two diagonal product actions, without asserting nonexistence for arbitrary invariant metrics. Abstract graph restrictions, smooth realizability, and positive-curvature realizability are kept distinct. Some of the results in this paper are discovered by ChatGPT-6-astra and verified by the authors.
\end{abstract}
\maketitle
\setcounter{tocdepth}{1}
\tableofcontents

% ===== 01_introduction.tex =====
\section{Introduction}\label{sec:intro}

The relation between positive sectional curvature and topology remains one of the central questions in global Riemannian geometry. Compact rank-one symmetric spaces provide the basic examples. The homogeneous constructions of Wallach and the biquotient constructions of Eschenburg showed that this class is not exhausted by symmetric spaces \cite{Wallach,Eschenburg}. In dimension six the homogeneous flag manifold
\[
 \Wflag=\SU(3)/T^2
\]
and its inhomogeneous Eschenburg counterpart $\Eflag$ are particularly important: they have the same Betti numbers, but their integral intersection data and Pontryagin classes distinguish them. Thus even a sharp Euler-characteristic theorem does not settle the smooth classification problem. Surveys of the examples and construction techniques can be found in \cite{ZillerSurvey,WilkingSurvey}.

The symmetry program asks how much topology is forced when a positively curved metric has a prescribed isometry group. Grove's formulation \cite{GroveProgram} encompasses both large torus rank and small cohomogeneity. The maximal-rank theorem of Grove--Searle \cite{GroveSearleRank}, the almost maximal-rank result of Fang--Rong \cite{FangRong}, Wilking's connectedness and symmetry theorems \cite{WilkingConnected,WilkingSymmetry}, the four-dimensional results of Hsiang--Kleiner and Grove--Wilking \cite{HK,GroveWilking}, the five-dimensional result of Rong \cite{rong} and the recent ten-dimensional result of Krishnan--Wiemeler \cite{Anusha} demonstrate the strength of this approach. Nonabelian symmetry supplies different information from torus rank. A compact connected nonabelian Lie group contains a subgroup locally isomorphic to $\SU(2)$; its effective image is $\SU(2)$ or $\SO(3)$. These groups have rank one, so their presence alone is not a large-torus-rank hypothesis. In dimension six, however, a generically free action of either group has cohomogeneity three. The geometry of a three-dimensional quotient and the topology of six-manifolds can therefore be used together.

Related five-dimensional work includes Simas's study of nonnegatively curved five-manifolds with nonabelian symmetry \cite{Simas}. Liu subsequently established an Euler-characteristic theorem, determined the nontrivial-principal-isotropy actions, classified several fixed-point and exceptional-free cases, and analyzed the remaining orbit strata in dimension six \cite{LiuPublished,LiuCorrected}. His corrected manuscript is our starting point. In particular, the polyhedral exceptional isotropy exhibited there must be retained: the irreducible action on the round harmonic-cubic six-sphere has tetrahedral isotropy. An argument that excludes every noncyclic exceptional group would therefore eliminate a genuine positively curved example.

Several recent developments clarify the setting of the present work. Kennard--Wiemeler--Wilking's splitting methods and their connection between torus representations and regular matroids provide new curvature obstructions under higher-rank torus symmetry \cite{KWWsplit,KWWmatroids}. These results concern a different symmetry regime from the rank-one actions considered here. More directly, Khalili Samani--Radeschi proved rational ellipticity for positively curved cohomogeneity-three manifolds with boundaryless orbit space and established a global catalogue of quotient graphs, together with branched-cover and unknotting tools \cite{KSR}. Their results turn local slice restrictions into a finite list of abstract graph shapes. The double disk-bundle theorem of DeVito--Galaz-Garc\'ia--Kerin supplies a complementary route to rational ellipticity for the ball-quotient configurations \cite{DGGK}. We use these theorems as inputs.

Recent construction preprints also counsel care in formulating product obstructions. In August 2026 Brendle--Hung reported a positive-sectional-curvature construction on $S^2\times S^2$ using a higher-order perturbation of a Cheeger deformation \cite{BH}, and Guo--Fang--Lu reported an analogous construction on $S^2\times S^3$ \cite{GFL1}. In September 2026   Guo--Fang--Lu reported a construction on $S^3\times S^3$ \cite{GFL2}, and He--Liu--Yau reported constructions on all smooth homotopy seven-spheres \cite{HLY}. These are cited here as recent preprints, not as ingredients in any proof below. They continue a construction direction exemplified earlier by Dearricott and Grove--Verdiani--Ziller \cite{Dearricott,GVZ}. Neither report decides positive curvature on $S^2\times S^4$ with either of the specific rank-one actions studied here. In particular, the Euler-characteristic form of the Hopf conjecture, the historical product question, and a prescribed-action metric-existence problem are three different assertions. The last must be addressed by an obstruction to all invariant metrics, not merely to product or connection metrics.

\subsection{Standing hypotheses and notation}
Unless a statement explicitly concerns smooth realizability without curvature, we assume
\begin{equation}\label{eq:H}
 \begin{gathered}
 M^6\text{ is a closed, smooth, simply connected manifold},\\
 g\text{ is a Riemannian metric with }\Sec_g>0,\\
 G=\SO(3)\text{ or }\SU(2)\text{ acts effectively and isometrically on }(M,g).
 \end{gathered}
\end{equation}
Write $X=M/G$ and $\pi:M\to X$. A singular orbit has dimension smaller than a principal orbit; an exceptional orbit has principal dimension but larger isotropy. Thus a finite-isotropy branching vertex is exceptional, not singular in this terminology. Let $\Sigma\subset X$ be the image of all nonprincipal orbits. An edge labelled $m$ means cyclic isotropy $C_m$ of order $m\ge2$. A label $1$ is omitted, since it describes principal points. We denote by $D_m\subset\SO(3)$ the rotational dihedral group of order $2m$.

For $G=\SO(3)$ put $T=\SO(2)$; for $G=\SU(2)$ put
\[
 T=\{\diag(z,z^{-1}):z\in S^1\}.
\]
The real oriented weight-$j$ plane is denoted by $\C_j$. Unless signs are needed for a normal Euler class, slice weights are positive unordered pairs. Cohomology has integral coefficients unless otherwise indicated. A statement about a diffeomorphism is an ordinary smooth classification unless equivariance is explicitly specified. All references to Liu's theorem, remark, or figure numbers mean the corrected version \cite{LiuCorrected}, not the earlier published text \cite{LiuPublished}.

\subsection{The classification results}
The following statements summarize the main conclusions. Their proofs occupy Sections~\ref{sec:so3s3}--\ref{sec:su2s3}; the models and all parameters are defined there.

\begin{theorem}[$\SO(3)$ and a spherical quotient]\label{thm:mainSO3}
Assume \eqref{eq:H}, $G=\SO(3)$, and $X\cong S^3$. Every singular isotropy group is $T$, and the number $r$ of singular orbits satisfies $1\le r\le3$ and $\chi(M)=2r$. The abstract weighted graph $\Sigma$ is one of the nine families in Tables~\ref{tab:graphs} and \ref{tab:lollislices}.

If every exceptional isotropy group is cyclic, then $r=2$ or $3$. For $r=2$ the action is equivariantly diffeomorphic to
\[
 N^G_{p,q}=G\times_T S(\R\oplus\C_p\oplus\C_q),
 \qquad p,q\ge1,\quad\gcd(p,q)=1,
\]
and $M\cong S^2\times S^4$. For $r=3$ it is equivariantly diffeomorphic to
\[
 P_{a,b}=G\times_T\CP^2
 =\mathbb P(\one\oplus L^{2a}\oplus L^{2(a+b)})\longrightarrow S^2,
 \qquad a,b\ge1,\quad\gcd(a,b)=1,
\]
where $c_1(L)$ generates $H^2(S^2)$. The underlying manifold is the product $S^2\times\CP^2$ precisely when $a\equiv b\pmod3$; otherwise it is the nontrivial unoriented diffeomorphism type of a linear $\CP^2$-bundle over $S^2$.

The remaining graphs are a dihedral $Y$, five types of lollipop, and one $D_3$--$A_4$ theta graph. Their rational intersection forms and Pontryagin covectors are given in Section~\ref{sec:branchinvariants}. The nine-family list is necessary; it is not a positive-curvature realization theorem or a classification of all spatial embeddings.
\end{theorem}

\begin{theorem}[$\SO(3)$ and a ball quotient]\label{thm:mainball}
Assume \eqref{eq:H}, $G=\SO(3)$, and $X\cong B^3$.
\begin{enumerate}
\item In Liu's Figure~2(b), the underlying manifold is
\[
 E_n=S(\xi_n)\longrightarrow S^4,\qquad p_1(\xi_n)=4n u,
 \qquad n\in\Z,
\]
where $\xi_n$ is an oriented three-plane bundle and $u$ generates $H^4(S^4)$. The ordinary diffeomorphism type is determined by $|n|$; $E_0=S^2\times S^4$ and $E_1=\CP^3$.
\item In Figure~2(c), the smooth action is obtained by gluing a weight-$(1,1)$ singular-orbit tube to a two-disk bundle over $S^2\times S^2$ whose Euler class is $2a+kb$. The gluing exists exactly for odd $k$ and has no further equivariant attaching parameter. The resulting manifolds $M_k$ satisfy
\[
 M_k\cong M_\ell\quad\Longleftrightarrow\quad |k|=|\ell|.
\]
They are non-spin, have torsion-free homology, and their full cubic and Pontryagin data are given in \eqref{eq:Mkcubic}--\eqref{eq:Mkp1}.
\item Figure~2(e) cannot occur. More generally, if $f$ is the number of isolated boundary fixed points with tangent representation $\R^3\oplus\R^3$, and $r$ is the number of interior singular points, then $f+2r\le3$.
\item In Figure~2(f), $M$ is diffeomorphic to $\CP^3$. The known projective $\SO(3)$ action realizes this configuration with positive curvature.
\end{enumerate}
Figures~2(a) and~2(d) are the cases treated by Liu; in~(d) the underlying manifold is diffeomorphic, not merely homeomorphic, to the standard $S^6$. The smooth families in (i)--(ii) do not assert positive-curvature realizability for every parameter.
\end{theorem}

\begin{theorem}[$\SU(2)$ and a spherical quotient]\label{thm:mainSU2}
Assume \eqref{eq:H}, $G=\SU(2)$, and $X\cong S^3$. Then the action is one of the following equivariant diffeomorphism types:
\begin{enumerate}
\item the linear action on $S(\C^2\oplus\mathfrak{su}(2))\cong S^6$;
\item $N^G_{1,1}$ or $N^G_{2j-1,2j+1}$, $j\ge1$, on $S^2\times S^4$;
\item the standard block action on $\Wflag$;
\item the commuting right block action on $\Eflag$.
\end{enumerate}
The only genuine exceptional stratum in the three-singular-orbit case is one $C_3$ interval, and its presence characterizes (iv). Consequently the middle panel of Liu's Figure~4 has labels $(1,1)$ or $(1,3)$, and the right panel has $(1,1,1)$ or $(1,1,3)$, up to permutation and deletion of labels~$1$. The surviving three-orbit patterns have positive-curvature realizations. For the two-orbit patterns, the displayed weights are necessary conditions; invariant strict positivity is not established here.
\end{theorem}

The new geometric input in Theorem~\ref{thm:mainSU2} is a matching condition between an exceptional slice and the adjoint representation: if a $C_k$ edge has transverse character $s$, positive curvature forces $s\equiv\pm2\pmod k$. The proof uses Killing Jacobi fields and the fact, proved by Liu, that the singular spheres are fixed components of the center and hence totally geodesic.

\begin{corollary}\label{cor:globalelliptic}
Every manifold satisfying \eqref{eq:H} is rationally elliptic. In particular,
\[
 H^{\mathrm{odd}}(M;\Q)=0,\qquad
 b_2(M)=b_4(M)=\frac{\chi(M)-2}{2}\le2.
\]
The possible real homotopy types are $S^6$, $\CP^3$, $S^2\times S^4$, $S^2\times\CP^2$, $\Wflag$, and $\CP^3\#\CP^3$.
\end{corollary}
This is a real-homotopy conclusion, not a replacement for integral reconstruction. In particular, the Eschenburg flag has the real homotopy type of $\CP^3\#\CP^3$ but is not thereby identified with that connected sum.

\subsection{Relation to Liu's results}
Liu's corrected Theorem~1.1 gives the Euler characteristic, the principal-isotropy alternatives, restrictions on exceptional groups, and the bound on singular orbits. His Theorem~1.2 already classifies fixed-point-free $\SU(2)$ actions without exceptional orbits, and his Theorem~1.3 treats the exceptional-free mixed-boundary ball configuration. None of these results is claimed as new here.

The extensions are the weighted and integral refinements. For $\SO(3)$ with quotient $S^3$, the local orientation argument eliminates the disconnected-singular-isotropy alternative left in Liu's Proposition~4.6. The global graph theorem of \cite{KSR}, combined with edge congruences and primitive normal Euler classes, gives the nine-family catalogue and the cyclic equivariant reconstruction. For $\SO(3)$ with quotient $B^3$, the sphere-bundle description suggested in Remark~4.5 is proved, the one-interior-orbit family is completely distinguished, the two-fixed-point/interior-orbit configuration is excluded by doubling, and the remaining exceptional configuration is recognized smoothly as $\CP^3$. For $\SU(2)$, the new weight restriction and normal-line reconstruction extend Liu's exceptional-free theorem to the exceptional cases in Figure~4.

The discussion of invariant product metrics is secondary. It establishes nonexistence within several broad geometric subclasses and identifies an action-specific boundary inequality for the $\SO(3)$ action. It neither removes $S^2\times S^4$ from the smooth candidate lists nor inserts unproved submersion or additional-symmetry assumptions into the classification arguments.

% ===== 02_tools.tex =====
\section{Common geometric and topological tools}\label{sec:tools}

\subsection{Results used from the literature}
We collect the precise inputs to separate the established theory from the reconstruction arguments below. The slice theorem and the classification of orbit bundles by normalizers are used in the form given by Bredon \cite[Chapters IV--V]{Bredon}. The quotient of an isometric compact-group action on a manifold with $\Sec\ge\kappa$ is an Alexandrov space with curvature at least $\kappa$; spaces of directions are the corresponding unit normal-sphere quotients \cite{AlexandrinoBettiol}.

Liu's corrected Theorems~1.1--1.3, Theorem~3.1, and Proposition~4.9 give the following reductions \cite{LiuCorrected}. The quotient is $S^3$, $B^3$, or $B^4$. Apart from the linear $\SO(3)$ actions on $S(\R^3\oplus\R^4_{\mathrm{triv}})$ and $S(V_5\oplus\R^2_{\mathrm{triv}})$, principal isotropy is trivial. The Euler characteristic is $2$, $4$, or $6$. In cohomogeneity three every exceptional stratum has dimension at most one and the closure of each exceptional component meets a singular stratum. With quotient $S^3$ there are at most three singular orbits. For $\SU(2)$ without fixed points, these are $G/T$ spheres, their union is the fixed set of the center, and exceptional isotropy is cyclic of odd order. Actions with nonempty fixed set, other than the finite-fixed-set $\SO(3)$ ball cases, are the linear cases identified by Liu using the low fixed-point-cohomogeneity classification of Grove--Kim and Grove--Searle \cite{GroveKim,GroveSearleLow}.

For a boundaryless quotient we use three results of Khalili Samani--Radeschi \cite{KSR}: their Theorem~A gives rational ellipticity; Proposition~5.1 gives the eighteen possible abstract quotient graphs; and their branched-cover and unknotting arguments, including Proposition~2.14 and Lemma~6.2, give the orbifold products needed in the cyclic reconstructions. The graph catalogue is a necessary list and is used as such. We do not assume that every listed graph is metrically realizable.

The boundary-soul theorem of Grove--Kim \cite[Theorem 1.2]{GroveKim}, with its strata interpreted as in Liu's Theorem~2.9, controls a positively curved quotient with boundary. In particular there is a unique farthest point from the boundary, and any isolated interior stratum is that point. We also use Petrunin's metric doubling theorem \cite[Theorem 2.1]{Petrunin}. Doubling an Alexandrov space of curvature at least $\kappa$ along its whole boundary preserves that lower bound. At a boundary point its new link is the double of the old link.

The extent inequality is used in the following form: four distinct points of a positively curved Alexandrov space cannot all have spaces of directions with three-extent at most $\pi/3$ \cite{GroveMarkvorsen,LiuCorrected}. The weighted Hopf quotients $S^3/S^1$ arising from two nonzero circle weights satisfy that bound \cite{HK}; the round Hopf quotient is $S^2(1/2)$. These general comparison results will be cited, not reproved.

For a closed simply connected six-manifold that is a union of two disk bundles, the rational-ellipticity theorem of DeVito--Galaz-Garc\'ia--Kerin requires $b_2\le3$ \cite[Theorem A]{DGGK}. This numerical hypothesis is essential in the general theorem. The decompositions used here have $b_2=1$ or $2$ and therefore satisfy it.

\subsection{Rational and integral recognition}
A rationally elliptic space with positive Euler characteristic has no odd rational cohomology \cite{FHT}. Together with Poincar\'e duality and Liu's Euler-characteristic theorem, this proves, for every spherical quotient,
\begin{equation}\label{eq:betti}
 H^{\mathrm{odd}}(M;\Q)=0,\qquad b_2=b_4=(\chi(M)-2)/2.
\end{equation}
Herrmann's classification of rationally elliptic six-manifolds \cite{Herrmann} gives the real homotopy types listed in Corollary~\ref{cor:globalelliptic}. In rank two they can be read from the binary cubic $v\mapsto\int_Mv^3$: a double and a simple real root gives the $S^2\times\CP^2$ type; three distinct real roots give the flag type; and one real root gives the $\CP^3\#\CP^3$ type. The zero cubic and a pure cube do not give elliptic six-manifolds with $b_2=2$, $b_3=0$. These statements classify real, not integral, data.

For integral recognition we use the Wall--Jupp theorem \cite{Wall,Jupp}. A closed simply connected oriented smooth six-manifold with torsion-free homology is determined up to oriented diffeomorphism by the isomorphism class of
\begin{equation}\label{eq:WJ}
 \bigl(H^2(M;\Z),\ b_3(M),\ \mu_M,\ w_2(M),\ p_M\bigr),
\end{equation}
where
\[
 \mu_M(x,y,z)=\int_Mxyz,\qquad p_M(x)=\int_Mp_1(M)x.
\]
The isomorphism must respect the integral lattice and the reduction of $w_2$ modulo two. We use this classification only after proving torsion-freeness. More general six-dimensional classification results, such as \cite{Zubr}, involve additional data; a rational cohomology computation is not a substitute for them. Finally, an integral homology six-sphere that is simply connected is a homotopy six-sphere, and $\Theta_6=0$ by Kervaire--Milnor \cite{KM}. This gives the ordinary smooth upgrade of Liu's topological sphere case.

\subsection{Weight normalization and singular-orbit tubes}
Choose orientations so that the homogeneous weight-one line over $G/T=S^2$ has Chern number
\begin{equation}\label{eq:delta}
 d_G=\begin{cases}1,&G=\SU(2),\\2,&G=\SO(3).\end{cases}
\end{equation}
Thus the tangent representation of $T$ on $G/T$ has weight $2/d_G$. For positive coprime integers $p,q$ write
\begin{equation}\label{eq:tube}
 U^G_{p,q}=G\times_TD(\C_p\oplus\C_q),\qquad
 B^G_{p,q}=\partial U^G_{p,q}=G\times_TS(\C_p\oplus\C_q).
\end{equation}
The orbit space of the boundary is the spindle $S^2(p,q)$, with cone points of orders $p,q$ when these exceed one. The underlying oriented four-plane bundle over $G/T$ has clutching class $d_G(p+q)\bmod2$ in $\pi_1(\SO(4))=\Z/2$. Consequently, for every pair in the $\SO(3)$ discussion and for odd pairs in the $\SU(2)$ discussion,
\begin{equation}\label{eq:tubetrivial}
 U^G_{p,q}\cong S^2\times D^4,\qquad B^G_{p,q}\cong S^2\times S^3
\end{equation}
as ordinary manifolds. This does not make the equivariant boundary bundle trivial.

\begin{lemma}[Equivariant boundary extension]\label{lem:extension}
For the tubes in \eqref{eq:tubetrivial}, every $G$-equivariant diffeomorphism of $B^G_{p,q}$ is equivariantly isotopic to the restriction of a diffeomorphism of $U^G_{p,q}$. In particular, the isotopy itself can be absorbed in a collar, so an attaching map to such a tube introduces no additional parameter once the other piece and its boundary identification have been fixed.
\end{lemma}
\begin{proof}
The boundary is a locally free $G$-space over $S^2(p,q)$. An induced quotient diffeomorphism preserves the cone orders. An orientation-preserving orbifold diffeomorphism of this sphere is isotopic to the identity: remove invariant disks at its at most two cone points, use the marked version of the sphere isotopy theorem \cite{Smale}, and perform the isotopy in the cyclic disk charts. There are no two distinct equal nontrivial orders, since $\gcd(p,q)=1$. Lifting this isotopy with an invariant connection reduces the problem to an equivariant map over the identity.

Such maps are sections of the adjoint orbibundle. In a chart with isotropy $C_m$, the value at the chart center lies in its centralizer. For $G=\SU(2)$, every nontrivial $m$ here is odd and its centralizer is a connected circle. Move these values to the identity through their centralizers. Obstruction theory on the remaining two-dimensional orbifold charts now reduces the section to the identity: $\pi_1(\SU(2))=\pi_2(\SU(2))=0$.

For $G=\SO(3)$ one must account for the disconnected centralizer of a half-turn. Consider the fiberwise double cover of the adjoint bundle with fiber $\SU(2)$. The obstruction to lifting a section is a class in
\[
 H^1_{\mathrm{orb}}(S^2(p,q);\Z/2)=0.
\]
Indeed the orbifold fundamental group has presentation
$\langle a,b\mid a^p=b^q=ab=1\rangle$ and is cyclic of order $\gcd(p,q)=1$. This also shows why a value in a disconnected half-turn centralizer cannot be chosen independently at a cone point. A lifted section has its values in the centralizers of the lifted cyclic subgroups of $\SU(2)$, which are connected. The preceding contraction and the same vanishing homotopy groups apply. This is the usual obstruction-theoretic description of orbibundle gauge sections; see \cite{AdemLeidaRuan} for the orbifold bundle language.

It remains to treat quotient orientation reversal. Let $w\in N_G(T)\setminus T$ and let $c$ be complex conjugation on both weight summands. Since $c(t v)=t^{-1}c(v)$ and $wtw^{-1}=t^{-1}$, the map
\[
 [g,v]\longmapsto[gw,c(v)]
\]
is a well-defined equivariant diffeomorphism of the whole tube. It reverses the orientation of its spindle quotient. Composing with it reduces the remaining case to the orientation-preserving one. The isotopies constructed above extend across an invariant collar. Smoothness at a cyclic chart center follows by working equivariantly in the covering disk throughout.
\end{proof}

The lifting argument is needed only for the nonfree spindle boundary. For a principal $G$-bundle over $S^2$, an automorphism over the identity is a section of the ordinary adjoint bundle. For $\SO(3)$ the lift obstruction is in $H^1(S^2;\Z/2)=0$, and $\pi_2(\SU(2))=0$ again makes the gauge section null-homotopic. For $\SU(2)$ this is immediate. We will use this free-boundary version in the ball cases.

\subsection{A closed exceptional interval and its normal bundle}
Suppose a closed exceptional interval of order $m$ has two circle-isotropy endpoints and no finite branching vertex. Its full inverse image, including the endpoints, is the smooth cohomogeneity-one four-manifold
\begin{equation}\label{eq:Am}
 A_m=G\times_TS(\R\oplus\C_m)
 =\mathbb P(\one\oplus L^{d_Gm})\longrightarrow S^2.
\end{equation}
Here the local axis planes in the two endpoint slices provide the smooth caps. Denote the fiber class by $f$ and the two orbit sections, with their $G/T$ orientations, by $S_-,S_+$. With a compatible choice of labeling,
\begin{equation}\label{eq:sections}
 [S_+]-[S_-]=d_Gm[f].
\end{equation}

\begin{lemma}[Normal characters]\label{lem:normal}
An oriented equivariant normal two-plane bundle $\nu$ of $A_m$ is determined by its oriented endpoint characters $u,v\in\Z$. They must satisfy $v\equiv u\pmod m$, and its Euler class obeys
\begin{equation}\label{eq:evaluations}
 \ang{e(\nu),[S_-]}=d_Gu,\quad
 \ang{e(\nu),[S_+]}=d_Gv,\quad
 \ang{e(\nu),[f]}=\frac{v-u}{m}.
\end{equation}
The unit circle bundle of $\nu$ is simply connected if and only if
\begin{equation}\label{eq:primitive}
 \gcd\left(d_Gu,\frac{v-u}{m}\right)=1.
\end{equation}
\end{lemma}
\begin{proof}
At a principal orbit $G/C_m$ the two endpoint representations restrict to the same complex character, which is exactly $u\equiv v\pmod m$. The normal bundle can be constructed on the two cohomogeneity-one disk bundles over the sections using these characters. On their common principal orbit any two complex-linear intertwiners differ by an element of $\C^*$, a connected group. Thus no additional clutching component remains. The restrictions to the sections have Chern numbers $d_Gu,d_Gv$, and \eqref{eq:sections} gives the fiber evaluation. Conversely these constructions realize every compatible pair. Since $A_m$ is simply connected, the homotopy sequence of its oriented circle bundle identifies its fundamental group with $\Z$ modulo the image of evaluation by $e(\nu)$. The section and fiber form an integral basis of $H_2(A_m)$, proving \eqref{eq:primitive}.
\end{proof}
A real transverse representation records a character only up to sign. The signs in Lemma~\ref{lem:normal}, in contrast, are fixed by an orientation of the normal bundle. This distinction is responsible for several restrictions that are invisible in an unsigned orbit diagram.

\subsection{Localization at the singular spheres}\label{sec:localization}
For a spherical quotient with connected singular isotropy, let $S_1,\ldots,S_r$ be the singular spheres, with slice weights $(p_i,q_i)$. Give each sphere the orientation from an identification with $\SU(2)/T$. For $x\in H^2(M;\Z)$ put $n_i(x)=\int_{S_i}x$. The $\SU(2)$ action used here is the given action when $G=\SU(2)$ and the double-cover action when $G=\SO(3)$.

\begin{proposition}\label{prop:localization}
There are signs $\eps_i\in\{1,-1\}$ determined by the tangent orientations such that
\begin{align}
 0&=\sum_i\frac{\eps_i n_i(x)}{p_iq_i},\label{eq:lin}\\
 \mu_M(x,y,z)&=\frac1{d_G^2}\sum_i
 \frac{\eps_i n_i(x)n_i(y)n_i(z)}{p_iq_i},\label{eq:cub}\\
 p_M(x)&=\sum_i\frac{\eps_i n_i(x)
 (4/d_G^2+p_i^2+q_i^2)}{p_iq_i}.\label{eq:pon}
\end{align}
The rational evaluation map $H^2(M;\Q)\to\Q^r$ is injective. When $\chi(M)=2r$, its image is exactly the hyperplane \eqref{eq:lin}.
\end{proposition}
\begin{proof}
We apply the equivariant localization theorem \cite{AB,BV}; only its specialization needs explanation. Since $H^j(B\SU(2);\Z)=0$ for $1\le j\le3$, every degree-two class has a unique $\SU(2)$-equivariant lift. At the two circle-fixed points of $S_i$, the lift restricts to $n_i(x)t$ and $-n_i(x)t$, where $t$ generates $H^2(BT)$. The tangent weights for this circle are
\[
 2,\ d_Gp_i,\ d_Gq_i.
\]
Their Euler products are $2\eps_i d_G^2p_iq_i t^3$ and its negative, while the equivariant first Pontryagin class is
$(4+d_G^2p_i^2+d_G^2q_i^2)t^2$ at either point. Applying localization to $x$, to $xyz$, and to $p_1x$ gives \eqref{eq:lin}--\eqref{eq:pon}. The integration of the degree-two lift is zero for dimensional reasons.

Equation~\eqref{eq:betti} gives equivariant formality over $\Q$. Restriction of equivariant cohomology to the circle-fixed set is therefore injective. If all $n_i(x)$ vanish, the unique equivariant lift vanishes and so does $x$. The image has dimension $b_2=r-1$, which proves the hyperplane assertion.
\end{proof}
For $\SO(3)$ the factor $1/4$ in \eqref{eq:cub} is indispensable. The proposition determines rational forms on an evaluation space. It does not identify the integral evaluation lattice or establish torsion-freeness of homology.

% ===== 03_so3_s3.tex =====
\section{\texorpdfstring{$\SO(3)$ actions with orbit space $S^3$}{SO(3) actions with orbit space S3}}\label{sec:so3s3}

Throughout this section $G=\SO(3)$, $X\cong S^3$, and \eqref{eq:H} holds. Principal isotropy is trivial by Liu. We first determine the local possibilities, then use the global graph theorem, and finally recover the smooth manifold in the cyclic cases.

\subsection{Connected singular isotropy and faithful finite slices}
\begin{lemma}\label{lem:connected}
Every singular isotropy group is $T=\SO(2)$. Its normal slice is $\C_p\oplus\C_q$ with nonzero coprime weights. If $r$ denotes the number of singular orbits, then
\[
 1\le r\le3,\qquad \chi(M)=2r.
\]
\end{lemma}
\begin{proof}
Fixed points would lie in the boundary of the quotient, by the six-dimensional fixed-point representations listed by Liu. A remaining singular isotropy group is $T$ or its normalizer $\OO(2)$. At an interior isolated singular orbit both circle weights on the four-dimensional slice are nonzero. An element of $\OO(2)\setminus T$ conjugates the circle to its inverse. On the slice it is anti-complex-linear on two complex dimensions, hence has positive real determinant, including when the two weights have equal absolute value. On the tangent plane of $G/\OO(2)=\RP^2$ it reverses orientation. It would therefore reverse the orientation of $M$. This is impossible for an element of the connected group $G$.

Thus every singular orbit is $G/T$. Slice effectiveness and trivial principal isotropy give $\gcd(p,q)=1$, while a zero weight would produce a boundary face. Every singular sphere contains two $T$-fixed points and no principal or exceptional orbit contains any. The Euler formula and singular-orbit bound of Liu now give the result.
\end{proof}
The determinant observation is the boundaryless counterpart of the orientation analysis already used in Liu's ball-quotient proof; it removes the disconnected-isotropy alternative from his Proposition~4.6.

At an exceptional orbit $G/\Gamma$, the slice has real dimension three. Its representation is faithful, since a nontrivial slice kernel would occur in principal isotropy nearby. It preserves orientation: both the adjoint action on the orbit tangent space and the ambient $G$ action do so. The finite rotation groups therefore give the following local incidence data \cite{Bredon,LiuCorrected}:
\begin{center}
\begin{tabular}{@{}ll@{}}
\toprule
Finite isotropy & Incident edge orders in the quotient slice\\
\midrule
$D_m$ & $2,2,m$\\
$A_4$ & $2,3,3$\\
$S_4$ & $2,3,4$\\
$A_5$ & $2,3,5$\\
\bottomrule
\end{tabular}
\end{center}
A cyclic edge has slice $\R\oplus\C_s$ for a generator of $C_m$, with $\gcd(s,m)=1$. At a $T$ vertex the two slice weights record the incident orders greater than one, so its valency is at most two. The representation of a finite group, not just its isomorphism type, remains part of the data. In particular the two faithful three-dimensional representations of $A_5$, relative to a fixed embedding, can have order-five transverse weights $\pm1$ and $\pm2$.

\begin{lemma}[Edge congruence]\label{lem:edgecong}
For an order-$m$ edge, the endpoint transverse characters satisfy $u\equiv\pm v\pmod m$. The only ambiguity in comparing the characters is inversion, not multiplication by an arbitrary unit of $\Z/m$.
\end{lemma}
\begin{proof}
Along a connected open edge, its normal slice representation is locally constant after identifying its isotropy subgroup. Conjugation in $\SO(3)$ between two identifications of a rotation subgroup is induced by an element of its normalizer and acts on the rotation circle by identity or inversion. Restricting the endpoint slice to this subgroup gives the asserted congruence.
\end{proof}

\subsection{The nine necessary weighted graphs}
Table~\ref{tab:graphs} gives the full abstract list. Filled vertices in Figure~\ref{fig:graphs} have circle isotropy; hollow vertices have noncyclic finite isotropy. The list does not discard the representation qualification in Table~\ref{tab:lollislices}.

\begin{table}[tbp]
\centering\small
\begin{tabularx}{\textwidth}{@{}clXcc@{}}
\toprule
Type & Shape & Weights and isotropy & $r$ & $\chi$\\
\midrule
I & Two points & Both slices $(1,1)$. &2&4\\
II & Interval & Edge $m\ge2$; endpoint slices $(1,m)$. &2&4\\
III & Bigon & Edges $m,n\ge2$, $\gcd(m,n)=1$; both slices $(m,n)$. &2&4\\
IV & Path & Edges $m,m+1$, $m\ge2$. &3&6\\
V & Triangle & Edges $a,b,a+b$, $a,b\ge2$, $\gcd(a,b)=1$. &3&6\\
VI & Point and interval & Isolated slice $(1,1)$ and one order-two interval. &3&6\\
VII & $Y$ & Finite vertex $D_m$; arms $2,2,m$, $m\ge2$. &3&6\\
VIII & Lollipop & Cycle orders $2,k$, stem $\ell$; finite vertex as in Table~\ref{tab:lollislices}. &2&4\\
IX & Subdivided theta & Finite vertices $D_3,A_4$; direct edges $2,3$; third path $D_3\xrightarrow{2}T\xrightarrow{3}A_4$. &1&2\\
\bottomrule
\end{tabularx}
\caption{Necessary abstract graphs for $\SO(3)$ with spherical quotient. A leaf on an order-$m$ edge has slice $(1,m)$; a two-valent circle vertex has its two incident orders as slice weights.}\label{tab:graphs}
\end{table}

\begin{table}[tbp]
\centering\small
\begin{tabularx}{\textwidth}{@{}lllX@{}}
\toprule
Finite vertex & Cycle & Stem & Required finite-slice character\\
\midrule
$D_k$, $k\ge3$ odd & $2,k$ & $2$ & Order-$k$ rotation weight $\pm2\bmod k$.\\
$A_4$ & $2,3$ & $3$ & Usual tetrahedral slice.\\
$S_4$ & $2,3$ & $4$ & Usual octahedral slice.\\
$A_5$ & $2,3$ & $5$ & Order-five weight $\pm1$.\\
$A_5$ & $2,5$ & $3$ & Order-five weight $\pm2$.\\
\bottomrule
\end{tabularx}
\caption{The five lollipop possibilities in VIII. The circle vertex on the cycle has slice $(2,k)$ and the leaf has slice $(1,\ell)$. In VII the order-$m$ transverse slice weight at $D_m$ is $\pm1$.}\label{tab:lollislices}
\end{table}

\begin{proposition}\label{prop:graphlist}
The graph $\Sigma$, after forgetting its embedding in $S^3$, belongs to exactly one family in Table~\ref{tab:graphs}, up to graph isomorphism and the indicated parameter symmetries.
\end{proposition}
\begin{proof}
We start with the eighteen graphs of \cite[Proposition 5.1]{KSR}. The specialization is explicit: graphs $1,2,3$ there are the cases with no edges; graphs $4,5,6,8,9$ give II, IV, VII, III, V; graphs $7,18$ contain a one-vertex loop; graphs $10,12,13,15,16$ have a three-valent non-orbifold vertex; graph $11$ gives VIII, graph $14$ the subdivided theta, and graph $17$ the point-plus-interval. A non-orbifold vertex here has circle isotropy and valency at most two, so the five three-valent cases are impossible. A one-vertex loop of order $m>1$ would force the slice pair $(m,m)$, contrary to coprimality.

For a path order its labels $a<b$. At the order-$b$ edge, the endpoint transverse characters are $1,a$, so $a\equiv\pm1\pmod b$. With $a\ge2$, this gives $a=b-1$. For a genuine triangle let $a<b<c$ be its labels. Along the $c$ edge, $a\equiv\pm b\pmod c$. The positive sign is impossible and the negative sign gives $c=a+b$. These are IV and V.

The global graph theorem gives an order-two edge on the cycle of a lollipop. Write its other cycle order as $k$ and its stem order as $\ell$. Coprimality at the cycle's circle vertex forces $k$ odd. Substitution of $(2,k,\ell)$ in the four finite-slice triples gives precisely Table~\ref{tab:lollislices}. Comparing the transverse character $2$ on the order-$k$ edge and the character $1$ on the stem gives the last column. The same comparison at the order-$m$ leaf of a $Y$ forces its finite-slice weight to be $\pm1$.

For the theta reduction we use the exceptional-track obstruction of \cite[Lemma 2.10]{KSR}. A complete quotient geodesic track that continues through finite vertices cannot avoid every singular vertex. In a finite slice, continuation pairs antipodal axis directions. At $D_m$, the two order-two arms are interchanged when $m$ is odd, while all arms reflect individually when $m$ is even. At $A_4$ the two order-three arms are interchanged and the order-two arm reflects. At $S_4$ and $A_5$ every arm reflects. These statements follow directly from the rotation-axis orbits on the unit two-sphere.

Let the two direct edges of a subdivided theta have orders $r,q$. If both are order two, they form a finite-only track, whether the endpoints interchange or reflect them. If they have different orders, a direct arm reflected at both endpoints likewise gives a forbidden track. The only available interchange mechanisms are the order-two pair at an odd dihedral vertex and the order-three pair at a tetrahedral vertex. Both must occur. Thus the direct orders are $2,3$, the finite vertices are $D_3,A_4$, and the remaining arms have orders $2,3$ and meet at the sole circle vertex. This gives IX.

Two topological restrictions complete the argument. If there are no edges, all slices are $(1,1)$. Remove $r$ tubes $U^G_{1,1}$. The remainder is a principal $\SO(3)$ bundle over $S^3$ with $r$ open balls removed. Each boundary restriction has nonzero $w_2$, because its total space $S^2\times S^3$ is simply connected. The sum of the boundary sphere classes is zero modulo two, so evaluating the bundle's $w_2$ gives $r\equiv0\pmod2$. Since $1\le r\le3$, only I remains.

Finally suppose an order-$m$ interval is disjoint from an isolated $(1,1)$ vertex. A tube about its inverse image $A_m$ and the isolated-orbit tube have a common boundary diffeomorphic to $S^2\times S^3$. The normal characters of $A_m$ have absolute values one. Set $u=1$, $v=\pm1$ in Lemma~\ref{lem:normal}. The choice $v=1$ has Euler evaluations $2,0$ and a non-simply-connected unit circle bundle. Hence $v=-1$ and $m\mid2$. Since $m>1$, $m=2$, which is VI. The decomposition used here follows by removing the two disjoint quotient balls and lifting the interval direction on their free complement $S^2\times[0,1]$.
\end{proof}

\begin{figure}[!htbp]
\centering
\begin{tikzpicture}[x=.94cm,y=.91cm]
\foreach \x in {0,5.05,10.1} {\foreach \y in {0,-4,-8} {\draw[gray!30] (\x-2.35,\y-1.7) rectangle (\x+2.35,\y+1.7);}}
\node[anchor=north west,font=\bfseries] at (-2.2,1.55) {I};
\node[sv] at (-.7,0) {};\node[sv] at (.7,0) {};
\node[font=\small] at (0,-1.1) {$(1,1)\qquad(1,1)$};
\node[anchor=north west,font=\bfseries] at (2.85,1.55) {II};
\draw[edge] (4.05,0)--node[lab,above]{$m$}(6.05,0);
\node[sv] at(4.05,0){};\node[sv]at(6.05,0){};
\node[font=\small] at(5.05,-1.1){$(1,m)\qquad(1,m)$};
\node[anchor=north west,font=\bfseries] at (7.9,1.55) {III};
\draw[edge] (9.1,0) to[bend left=40] node[lab,above]{$m$}(11.1,0);
\draw[edge] (9.1,0) to[bend right=40] node[lab,below]{$n$}(11.1,0);
\node[sv]at(9.1,0){};\node[sv]at(11.1,0){};
\node[anchor=north west,font=\bfseries] at (-2.2,-2.45) {IV};
\draw[edge](-1.5,-4)--node[lab,above]{$m$}(0,-4)--node[lab,above]{$m+1$}(1.5,-4);
\node[sv]at(-1.5,-4){};\node[sv]at(0,-4){};\node[sv]at(1.5,-4){};
\node[anchor=north west,font=\bfseries] at (2.85,-2.45) {V};
\draw[edge](3.95,-3.55)--node[lab,above]{$a$}(6.15,-3.55)--node[lab,right]{$b$}(5.05,-4.95)--node[lab,left]{$a+b$}(3.95,-3.55);
\node[sv]at(3.95,-3.55){};\node[sv]at(6.15,-3.55){};\node[sv]at(5.05,-4.95){};
\node[anchor=north west,font=\bfseries] at (7.9,-2.45) {VI};
\draw[edge](9.1,-3.7)--node[lab,above]{$2$}(11.1,-3.7);
\node[sv]at(9.1,-3.7){};\node[sv]at(11.1,-3.7){};\node[sv]at(10.1,-4.8){};
\node[anchor=north west,font=\bfseries] at (-2.2,-6.45) {VII};
\draw[edge](0,-8)--node[lab,above]{$2$}(-1.4,-7.2);
\draw[edge](0,-8)--node[lab,below]{$2$}(-1.4,-8.8);
\draw[edge](0,-8)--node[lab,above]{$m$}(1.5,-8);
\node[ev,label=below right:$D_m$]at(0,-8){};
\node[sv]at(-1.4,-7.2){};\node[sv]at(-1.4,-8.8){};\node[sv]at(1.5,-8){};
\node[anchor=north west,font=\bfseries] at (2.85,-6.45) {VIII};
\draw[edge](4.05,-7.65) to[bend left=38]node[lab,above]{$2$}(6.05,-7.65);
\draw[edge](4.05,-7.65) to[bend right=38]node[lab,below]{$k$}(6.05,-7.65);
\draw[edge](4.05,-7.65)--node[lab,left]{$\ell$}(4.95,-9.05);
\node[ev,label=left:$\Gamma$]at(4.05,-7.65){};
\node[sv]at(6.05,-7.65){};\node[sv]at(4.95,-9.05){};
\node[anchor=north west,font=\bfseries] at (7.9,-6.45) {IX};
\draw[edge](9,-7.6) to[bend left=38]node[lab,above]{$2$}(11.2,-7.6);
\draw[edge](9,-7.6)--node[lab,above]{$3$}(11.2,-7.6);
\draw[edge](9,-7.6)--node[lab,left]{$2$}(10.1,-9.05)--node[lab,right]{$3$}(11.2,-7.6);
\node[ev,label=left:$D_3$]at(9,-7.6){};\node[ev,label=right:$A_4$]at(11.2,-7.6){};
\node[sv]at(10.1,-9.05){};
\end{tikzpicture}
\caption{The nine necessary abstract weighted graphs. Filled vertices have circle isotropy and hollow vertices have noncyclic finite isotropy. The picture specifies incidence, not a classification of spatial embeddings or positive-curvature realizations.}\label{fig:graphs}
\end{figure}
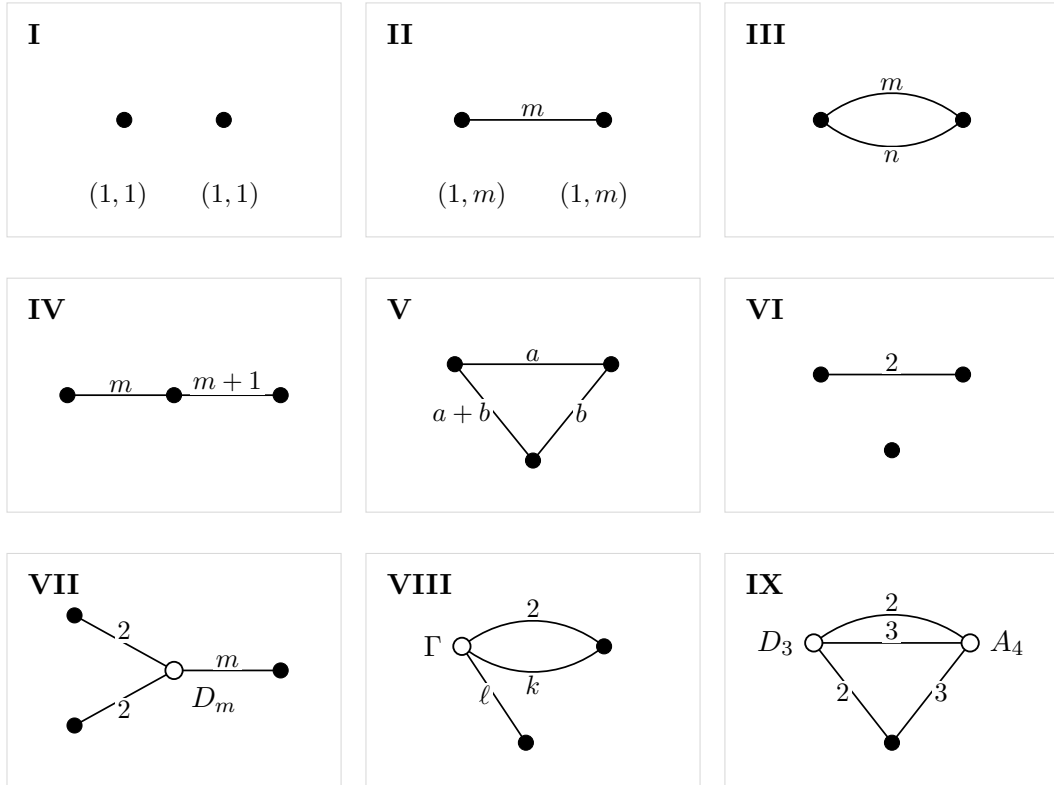
\FloatBarrier

\subsection{Smooth reconstruction in the cyclic cases}\label{sec:so3cyclic}
For I--III, remove the two singular-orbit tubes. In III the bigon cycle is unknotted by the branched-cover argument in \cite[Section 6, graph 8]{KSR}. In II a single tame arc with its endpoints free in $S^3$ can be straightened; in I there is no arc. The complement, with its quotient smooth structure, is $S^2(p,q)\times[0,1]$, using \cite[Lemma 6.2]{KSR}. An invariant connection lifts the interval direction and identifies the corresponding $G$-manifold with a product. Therefore
\[
 M=U^G_{p,q}\cup_\varphi U^G_{p,q}.
\]
Lemma~\ref{lem:extension} absorbs $\varphi$. The standard double is $N^G_{p,q}$. Its underlying oriented five-plane bundle over $S^2$ has $w_2=2(p+q)\bmod2=0$, so it is trivial. This proves $N^G_{p,q}\cong S^2\times S^4$.

For VI the argument at the end of Proposition~\ref{prop:graphlist} determines the normal line on $A_2$ uniquely, up to simultaneous reversal of the normal orientation and exchange of its sections. Its Euler evaluations are $2,-1$ on a section and a fiber. The last gluing is absorbed by $U^G_{1,1}$. The projective model $P_{1,1}$ realizes these data, proving uniqueness equivariantly.

For IV or V choose an edge of order $a\ge2$ whose endpoint transverse absolute weights are $b,a+b$, with $\gcd(a,b)=1$. The path case is $b=1$; in the triangle case both $a,b>1$. The third singular vertex has slice $(b,a+b)$. In the triangle case the cycle is unknotted by \cite[Section 6, graph 9]{KSR}; the orbifold product lemma then identifies the intervening region. Consequently
\begin{equation}\label{eq:projectiveglue}
 M=D(\nu A_a)\cup_{S^2\times S^3}U^G_{b,a+b}.
\end{equation}
Orient the normal line so that its first character is $u=b$. The second is $v=\pm(a+b)$. The positive choice is compatible and has fiber Euler number one. The negative choice requires $a\mid2b$, hence $a=2$. Then $b$ is odd, and both Euler evaluations $2b$ and $-(b+1)$ are even. This contradicts the simple connectedness of the common boundary in \eqref{eq:projectiveglue}. Thus the positive choice is forced. Lemmas~\ref{lem:normal} and~\ref{lem:extension} now determine the normal line and remove the final gluing ambiguity.

For the induced projective action
\begin{equation}\label{eq:projaction}
 z[z_0:z_1:z_2]=[z_0:z^a z_1:z^{a+b}z_2],
 \qquad P_{a,b}=G\times_T\CP^2,
\end{equation}
the three coordinate fixed points have slice pairs $(a,a+b)$, $(a,b)$, and $(b,a+b)$. Coordinate projective lines give the three possible edges, with labels one omitted. Thus $P_{a,b}$ realizes precisely the forced gluing data. This proves the equivariant assertions of Theorem~\ref{thm:mainSO3} for I--VI. The uniqueness argument uses the primitive normal Euler class and the boundary extension; it does not follow from the unsigned graph alone.

Let $u=c_1(L)$ on the base and $h=c_1(\mathcal O(1))$ on the projectivization, where $\mathbb P$ parametrizes lines. Put $c=2(2a+b)$. The projective-bundle formula and the vertical Euler sequence \cite{MS} give
\begin{align}
 H^*(P_{a,b};\Z)&=\Z[u,h]/(u^2,h^3+cuh^2),\quad |u|=|h|=2,
 \qquad\int_{P_{a,b}}uh^2=1,\label{eq:Pring}\\
 H^{\mathrm{odd}}(P_{a,b};\Z)&=0,\qquad
 w_2(P_{a,b})=\bar h,\qquad
 p_1(P_{a,b})=3h^2+2cuh.\label{eq:Pclasses}
\end{align}
Indeed the stable vertical tangent roots are $h,h+2au,h+2(a+b)u$. Summing their squares gives \eqref{eq:Pclasses}. Linear projective bundles over $S^2$ are classified by $\pi_1(PU(3))=\Z/3$; their clutching invariant here is
\[
 c\equiv a-b\pmod3.
\]
The two nonzero classes are interchanged by complex conjugation. To see that neither is diffeomorphic to the zero class, note that the primitive square-zero class in \eqref{eq:Pring} is $u$ up to sign. Changing the other generator to $\pm h+ju$ changes $c$ only by sign and a multiple of three. The two claimed ordinary diffeomorphism types therefore are distinct. All these cyclic models admit invariant nonnegative curvature, obtained from a bi-invariant metric on $G$ and a round or Fubini--Study metric by the free diagonal $T$ quotient; strict positivity is not a consequence of this construction.

\subsection{A smooth family with dihedral branching}
\begin{proposition}\label{prop:dihedralfamily}
For each odd $k\ge3$ there is a smooth simply connected $\SO(3)$-manifold with type-VIII graph $(D_k;2,k;2)$ whose underlying manifold is
\[
 E_n=S(\xi_n)\longrightarrow S^4,\qquad
 n=\frac{k^2-1}{8},\qquad p_1(\xi_n)=4n u.
\]
This is a smooth realization and does not assert positive curvature for every $k$.
\end{proposition}
\begin{proof}
Use the cohomogeneity-one action on $S^4\subset\operatorname{Sym}^2_0(\R^3)$, with principal isotropy $D_2$ and singular isotropy $K_-,K_+\cong\OO(2)$ about different coordinate axes. Define homomorphisms into a second $\SO(3)$ by multiplying the circle angle by $k$ on $K_-$ and using the ordinary inclusion on $K_+$. Send the reflection in each normalizer to the corresponding half-turn. Because $k$ is odd, the homomorphisms agree on $D_2$. The equivariant principal-bundle construction of Grove--Ziller \cite[Proposition 1.6]{GZ} gives a principal $\SO(3)$ bundle over $S^4$ with a commuting lifted $G$ action. Take its associated two-sphere bundle.

Over the principal base orbit the three axes of $D_2$ give three order-two arcs. At the weight-$k$ end the fiber poles produce a circle orbit with slice $(2,k)$, while the equator produces $D_k$. At the weight-one end the poles give slice $(1,2)$ and the equator has cyclic isotropy. These arcs assemble into the required lollipop; at its $D_k$ vertex the order-$k$ transverse character is $\pm2$.

At the two fixed points of a rotation circle on the base $S^4$, the tangent weights are $(1,2)$ with opposite Euler signs. The associated three-plane fiber weights are $k$ and $1$. Localization gives
\[
 \int_{S^4}p_1(\xi_n)=\frac{k^2-1}{2}=4n.
\]
The total sphere bundle is simply connected by the homotopy sequence. This proves the assertion, and its integral invariants are computed in \eqref{eq:Eninvariants} below.
\end{proof}

\subsection{What the branching graphs determine}\label{sec:branchinvariants}
In type VIII let $x$ generate $H^2(M;\Z)\cong\Z$, and set $g_0=\gcd(\ell,2k)$. The evaluation relation \eqref{eq:lin} permits compatible orientations and a positive integer $t$ such that
\[
 n_{\mathrm{leaf}}(x)=\frac{\ell t}{g_0},\qquad
 n_{\mathrm{cycle}}(x)=\frac{2kt}{g_0},
\]
with orientation signs $-,+$. Equations~\eqref{eq:cub}--\eqref{eq:pon} yield
\begin{equation}\label{eq:lollinumbers}
 d:=\int_Mx^3=\frac{(4k^2-\ell^2)t^3}{4g_0^3},\qquad
 P:=\int_Mp_1(M)x=\frac{(k^2+3-\ell^2)t}{g_0}.
\end{equation}
Since $2k>\ell$ in every permitted row, $d>0$. Thus these manifolds have rational cohomology $\Q[x]/(x^4)$, not the rational ring of $S^2\times S^4$.

\begin{table}[tbp]
\centering\small
\begin{tabular}{@{}llll@{}}
\toprule
Finite vertex & Parameter & $d$ & $P$\\
\midrule
$D_k$, $k$ odd & $t\ge1$ & $(k^2-1)t^3/8$ & $(k^2-1)t/2$\\
$A_4$ & $t=2h$, $h\ge1$ & $2h^3$ & $2h$\\
$S_4$ & $t=2h$, $h\ge1$ & $5h^3$ & $-4h$\\
$A_5$, $(k,\ell)=(3,5)$ & $t=2h$, $h\ge1$ & $22h^3$ & $-26h$\\
$A_5$, $(k,\ell)=(5,3)$ & $t=2h$, $h\ge1$ & $182h^3$ & $38h$\\
\bottomrule
\end{tabular}
\caption{Necessary characteristic numbers for lollipop graphs. Integrality forces $t$ even in the final four rows. These are arithmetic restrictions, not an existence list.}\label{tab:lollinumbers}
\end{table}
A primitive degree-two class need not have a primitive vector of evaluations on the singular spheres. Thus $t$ cannot be set equal to one without an integral lattice computation. Proposition~\ref{prop:dihedralfamily} realizes $t=1$ in the dihedral row. If homology is torsion-free, Wall--Jupp further requires
\begin{equation}\label{eq:rankonecong}
 \begin{cases}
 P\equiv4d\pmod{24},&w_2=0,\\
 d\equiv0\pmod2,\quad P\equiv d\pmod{48},&w_2=\bar x.
 \end{cases}
\end{equation}
These congruences are tests, not proofs that an arbitrary listed integer is realized.

For a type-VII $Y$, absorb the orientation signs into the evaluations $u,v,w$ on the two order-two leaves and the order-$m$ leaf. Then
\[
 \frac u2+\frac v2+\frac wm=0,\qquad w=-\frac m2(u+v).
\]
On this rational plane, localization gives
\begin{align}
 \mu(x,x,x)&=\frac{u^3+v^3}{8}-\frac{m^2(u+v)^3}{32},\label{eq:Ycubic}\\
 p_M(x)&=\frac{4-m^2}{2}(u+v).\label{eq:Ypon}
\end{align}
In the coordinates $s=u+v$, $d=u-v$, the cubic is
\[
 \frac1{32}s\bigl(3d^2-(m^2-1)s^2\bigr).
\]
It has three distinct real linear factors for every $m\ge2$. Hence these manifolds have the real homotopy type of $\Wflag$, although the rational Pontryagin class vanishes only when $m=2$. In particular $m>2$ cannot give the ordinary flag manifold. The integral evaluation lattice, torsion, and $w_2$ are not determined by this rational calculation.

For IX, $r=1$ gives the rational cohomology of $S^6$. If an integral gluing computation proves that homology is torsion-free, the manifold is diffeomorphic to $S^6$ by \cite{KM}. This conditional recognition does not suppress possible torsion.

Three known positively curved actions check the catalogue. The irreducible action on $S(\mathcal H_3(\R^3))$ has one circle orbit with slice $(2,3)$, so its graph is IX; Liu's harmonic cubic $xyz$ supplies the $A_4$ orbit. The projectivized irreducible action on $\CP^3=\mathbb P(\operatorname{Sym}^3\C^2)$ has slices $(2,3)$ and $(1,2)$ and graph VIII with $D_3$. The left $\SO(3)$ action on $\Wflag$ has graph VII with $m=2$: its finite stabilizers are intersections with tori, so the noncyclic one is the abelian group $D_2$ \cite[Examples 5.2--5.4]{LiuCorrected}. These identifications verify the examples; they do not prove that a branching graph determines a unique manifold.

% ===== 04_so3_b3.tex =====
\section{\texorpdfstring{$\SO(3)$ actions with orbit space $B^3$}{SO(3) actions with orbit space B3}}\label{sec:ball}

Let $G=\SO(3)$ and $X\cong B^3$. The action with nontrivial principal isotropy is already a linear action on $S^6$, by Liu's Theorem~3.2. We therefore assume principal isotropy is trivial and the fixed set is finite or empty. Liu's Theorem~4.2 gives a boundary face with $T$ isotropy, at most one interior singular orbit, and the alternatives in his Figure~2. We retain the letters of that figure, but describe every case in words in Table~\ref{tab:ball} so that no external illustration is needed.

\begin{table}[tbp]
\centering\small
\begin{tabularx}{\textwidth}{@{}cXXX@{}}
\toprule
Case & Boundary and interior & Smooth conclusion & Curvature conclusion\\
\midrule
(a) & Nontrivial principal isotropy; fixed circle and two faces. & Linear $S^6$; Liu. & Realized on the round sphere.\\
(b) & Uniform $T$ boundary; principal interior. & $E_n=S(\xi_n)\to S^4$, $n\in\Z$. & No parameter bound from the present argument.\\
(c) & Uniform $T$ boundary; one interior $T$ orbit. & $M_k$, $k$ odd, distinguished by $|k|$. & Smoothly realizable; strict positivity not decided.\\
(d) & Two boundary fixed points; no interior singular orbit. & Standard $S^6$. & Realized by the diagonal linear action.\\
(e) & Two boundary fixed points; one interior $T$ orbit. & Smoothly realized on a complex quadric. & Impossible under \eqref{eq:H}.\\
(f) & One boundary fixed point, one boundary $\OO(2)$ orbit, one interior $T$ orbit, and a $C_2$ interval. & $\CP^3$. & Realized by the real projective block action.\\
\bottomrule
\end{tabularx}
\caption{The $\SO(3)$ ball-quotient alternatives, with Liu's Figure~2 lettering. The topology of the entire boundary, not just a boundary face, is used.}\label{tab:ball}
\end{table}

\subsection{The full stratification and the boundary normal line}\label{sec:ballstrata}
At a point of a boundary face the slice is $\R^2_{\mathrm{triv}}\oplus\C_1$. At an isolated $G$-fixed point it is $\R^3\oplus\R^3$. Neither slice contains exceptional orbits. At the isolated boundary $\OO(2)$ orbit in~(f), the slice is
\begin{equation}\label{eq:O2slice}
 \R_{\det}\oplus\R_{\det}\oplus\C_{\mathrm{std}}.
\end{equation}
It has exactly one exceptional germ, of order two, arising from a reflection in $\OO(2)$ rather than from its identity component. These are the corrected slice alternatives in Liu's Theorem~2.11 and the proof of his Theorem~4.2 \cite{LiuCorrected}.

\begin{lemma}\label{lem:noextra}
In cases~(b)--(f), no exceptional stratum occurs other than the interval displayed in~(f). The interior slice is $(1,1)$ in~(c) and~(e), and $(1,2)$ in~(f).
\end{lemma}
\begin{proof}
An interior noncyclic finite isotropy group would give an isolated interior stratum. By the boundary-soul theorem it must be the soul. In cases with an interior singular orbit the soul is already that orbit, so such a finite stratum is impossible. In case~(b) there are no eligible boundary endpoints for any exceptional edge, and an exceptional component disjoint from singular strata is excluded by Liu's Proposition~3.6. The same argument applies to~(d), since its boundary fixed-point slices have no exceptional germs.

In~(c) and~(e) all exceptional endpoints would have to be the sole interior singular orbit. With no branching vertex available, an exceptional edge would then close to a loop at that orbit. Both slice weights would be its order, contradicting coprimality. Thus there are no exceptional edges and the two nonzero slice weights have absolute value one. In~(f), the boundary normalizer orbit supplies exactly one order-two germ. It must end at the soul, since a face or a fixed point cannot be an endpoint and there is no other boundary normalizer orbit. A second edge or an additional component would have no possible endpoint except a prohibited loop at the soul. The remaining interior weight is therefore one.
\end{proof}

When the boundary has a single orbit type, let $F=\pi^{-1}(\partial X)$. It is a smooth codimension-two submanifold. Its orbit bundle over $S^2$ has structure group $N_G(T)/T=\Z/2$ and is consequently trivial. Fix
\begin{equation}\label{eq:Fandnu}
 F=S^2_o\times S^2_b,\qquad
 e(\nu F)=2a+kb,
\end{equation}
where $a,b$ generate the cohomology of the orbit and quotient factors, respectively. The coefficient two is forced by the weight-one normal character on $G/T$. The integer $k$ is a genuine normal-bundle invariant. If $\eta$ is the homogeneous complex line over $S^2_o$ with $c_1(\eta)=2a$ and $c_1(L)=b$, then the equivariant normal line is
\[
 \nu_k=\eta\boxtimes L^k.
\]
This also constructs the full boundary tube $V_k=D(\nu_k)$.

\subsection{Uniform boundary with principal interior: the sphere bundles}
\begin{proposition}\label{prop:ballb}
In case~(b), $k=2n$ is even, and the underlying manifold is $E_n=S(\xi_n)\to S^4$ with $p_1(\xi_n)=4n u$. Its integral data, in compatible generators $x\in H^2(M)$ and $y\in H^4(M)$, are
\begin{equation}\label{eq:Eninvariants}
 \begin{gathered}
 H^{\mathrm{odd}}(M;\Z)=0,\quad H^2(M)=H^4(M)=\Z,\quad
 \int_Mxy=1,\quad x^2=ny,\\
 w_2(M)=0,\qquad p_1(M)=4ny.
 \end{gathered}
\end{equation}
Every integer $n$ is smoothly realizable in this stratum, and the ordinary diffeomorphism type is determined by $|n|$.
\end{proposition}
\begin{proof}
After removing the boundary collar, the principal region is $U=G\times D^3$. Thus
\[
 M=U\cup_WV_k,\qquad W=S(\nu_k)\cong G\times S^2.
\]
For a circle bundle over the simply connected base $F$, the fundamental group of its total space has order equal to the divisibility of its Euler class. Hence $\pi_1(W)=\Z/2$ forces $k=2n$. Put $c=a+nb$, so $e(\nu_k)=2c$.

The Gysin sequence gives $H^2(W)=\Z\oplus\Z/2$ and $H^4(W)=\Z/2$. The torsion generator from $H^2(U)=\Z/2$ restricts to the class of $c$ modulo $2c$. The restriction $H^3(U)\to H^3(W)$ is an isomorphism, since it restricts as an equivariant degree-one map on a principal $G$ orbit. Mayer--Vietoris now gives torsion-free cohomology in degrees $0,2,4,6$, one generator in each degree, and zero odd cohomology. Generators can be chosen with
\[
 x|_F=c,\qquad y|_F=2ab,\qquad\int_Mxy=1.
\]
For the last normalization, $\PD[F]=2x$ and $\int_Fy|_F=2$ imply $\int_Mxy=1$. Restriction to $F$ is injective in degree four. It follows from $c^2=2nab$ that $x^2=ny$.

Since $TM|_F=TF\oplus\nu_k$, $F$ is spin, and $e(\nu_k)=2c$, restriction gives $w_2(M)|_F=0$. Restriction in degree two modulo two is injective, so $w_2(M)=0$. Moreover $p_1(TF)=0$ and
\[
 p_1(M)|_F=e(\nu_k)^2=8nab=4n\,y|_F,
\]
which proves \eqref{eq:Eninvariants}.

Oriented three-plane bundles over $S^4$ are classified by an integer normalized by $p_1(\xi_n)=4n u$ \cite{MS,Husemoller}. Their sphere bundles have exactly \eqref{eq:Eninvariants}: lift the structure group to $\SU(2)$ and use its rank-two complex projectivization, or use the vertical Euler class $2x$ and the stable tangent splitting. Wall--Jupp identifies $M$ with $E_n$. The cubic and its integral divisibility distinguish $|n|$; simultaneously changing $x,y$ to $-x,-y$ identifies the systems for $n$ and $-n$.

For existence as a smooth action, the principal $G$ bundle $S(\nu_{2n})\to S^2_b$ has clutching class $2n\bmod2=0$. It can therefore be glued equivariantly to $G\times D^3$. Van Kampen gives simple connectivity, and the local models give exactly the strata in~(b). Automorphisms of the free boundary are absorbed using the free-boundary argument following Lemma~\ref{lem:extension}, together with extension of a sphere diffeomorphism over $D^3$. This proves smooth realizability, with no assertion about strictly positive metrics.
\end{proof}
The cases $n=0$ and $n=1$ are $S^2\times S^4$ and the twistor space $\CP^3$, respectively. Proposition~\ref{prop:ballb} proves the sphere-bundle description proposed in Liu's Remark~4.5. It does not show that positive curvature forces $|n|\le1$.

\subsection{Uniform boundary with one singular interior orbit}
The interior tube in~(c) is $U=U^G_{1,1}$ by Lemma~\ref{lem:noextra}. The region between it and the outer boundary tube is a principal bundle over $S^2\times[0,1]$, hence an equivariant product in the interval direction. It can be absorbed into a collar. Consequently the entire manifold is
\begin{equation}\label{eq:Mkdef}
 M_k=U^G_{1,1}\cup_\varphi D(\eta\boxtimes L^k).
\end{equation}

\begin{proposition}\label{prop:oddglue}
The gluing \eqref{eq:Mkdef} exists exactly when $k$ is odd. For each such $k$ it gives a closed simply connected smooth effective $G$-manifold with the complete stratification~(c). The equivariant attaching map introduces no further parameter.
\end{proposition}
\begin{proof}
The boundary $W_k=S(\eta\boxtimes L^k)$ is a free $G$-space over $S^2_b$. Local trivializations of $L^k$ show that its principal $G$-bundle clutching class is $k\bmod2$: the transition loop is the weight-$k$ rotation loop in $T\subset G$. On the other side,
\[
 \partial U^G_{1,1}=G\times_TS^3\longrightarrow S^3/T=S^2
\]
is the extension of the Hopf circle bundle along $T\to G$. Its clutching class is the nonzero element of $\pi_1(G)=\Z/2$. Therefore an equivariant boundary identification exists precisely for odd $k$.

For odd $k$, the Euler class $2a+kb$ is primitive and $W_k$ is simply connected. Both disk-bundle pieces in \eqref{eq:Mkdef} are simply connected, so van Kampen gives $\pi_1(M_k)=0$. The principal isotropy is trivial, and the stated slices contain no further exceptional strata. To compare two attaching maps, compose one with the inverse of the other on $\partial U^G_{1,1}$. Lemma~\ref{lem:extension} absorbs this boundary automorphism into that tube. This proves uniqueness for the fixed normal-bundle data.
\end{proof}

\begin{theorem}[Integral classification of the odd family]\label{thm:Mk}
For odd $k$, the homology of $M_k$ is torsion-free, $b_2=b_4=2$, and $H^{\mathrm{odd}}(M_k;\Z)=0$. There is a basis $x,y$ of $H^2(M_k;\Z)$ restricting to $a,b$ on $F$ such that
\begin{align}
 \int_{M_k}x^3&=\frac{k^3-k}{4},&
 \int_{M_k}x^2y&=\frac{1-k^2}{2},&
 \int_{M_k}xy^2&=k,&
 \int_{M_k}y^3&=-2,\label{eq:Mkcubic}\\
 w_2(M_k)&=\bar y,&&\PD[F]=2x+ky,\label{eq:Mkw2}\\
 \int_{M_k}p_1(M_k)x&=3k,&&
 \int_{M_k}p_1(M_k)y=-2.\label{eq:Mkp1}
\end{align}
Moreover $M_k\cong M_\ell$ if and only if $|k|=|\ell|$, and every $M_k$ is rationally elliptic.
\end{theorem}
\begin{proof}
The boundary is $S^2\times S^3$. In Mayer--Vietoris, the map $H^2(U)\to H^2(W_k)$ is an isomorphism, so restriction $H^2(M_k)\to H^2(F)$ is an isomorphism. The remaining terms give $H^3(M_k)=0$ and a short exact sequence
\[
 0\longrightarrow\Z\longrightarrow H^4(M_k)\longrightarrow\Z\longrightarrow0.
\]
Poincar\'e duality completes the additive computation. The restriction of $TM_k$ to $F$ and the normal Euler class give $w_2=\bar y$ and $\PD[F]=2x+ky$.

The full cubic requires more than restriction to $F$, since $H^4(M_k)\to H^4(F)$ is not injective. We compute it by the localization theorem \cite{AB,BV}. Let $u$ generate $H^2(BT;\Q)$ for the intrinsic rotation circle in $\SO(3)$. The fixed set consists of two copies $F_+,F_-$ of $S^2_b$ over the boundary and two points $q_+,q_-$ on the interior orbit. For the canonical rational $G$-equivariant lifts, the restrictions are
\begin{center}
\begin{tabular}{@{}ccccc@{}}
\toprule
 & $F_+$ & $F_-$ & $q_+$ & $q_-$\\
\midrule
$x$ & $u/2$ & $-u/2$ & $ku/2$ & $-ku/2$\\
$y$ & $b$ & $b$ & $-u$ & $u$\\
Euler class & $u(u+kb)$ & $u(u-kb)$ & $u^3$ & $-u^3$\\
\bottomrule
\end{tabular}
\end{center}
The point restrictions follow from the boundary relation $2a+kb=0$ in $H^2(W_k)$ and the fact that the two Weyl-related points have opposite restrictions. The point Euler sign is fixed by integrating the degree-four class $u^2y$: its total integral is zero. Thus these signs do not require an arbitrary extra orientation choice.

For clarity, the contributions from the two surfaces and from the two points, respectively, are
\begin{center}
\begin{tabular}{@{}crr@{}}
\toprule
Class & Surfaces & Points\\
\midrule
$x^3$ & $-k/4$ & $k^3/4$\\
$x^2y$ & $1/2$ & $-k^2/2$\\
$xy^2$ & $0$ & $k$\\
$y^3$ & $0$ & $-2$\\
\bottomrule
\end{tabular}
\end{center}
Their sums give \eqref{eq:Mkcubic}. At $F_\pm$ the equivariant Pontryagin class is $2u^2\pm2kub$, and at $q_\pm$ it is $3u^2$. The surface contributions to $p_1x,p_1y$ are $0,4$, while the point contributions are $3k,-6$. This proves \eqref{eq:Mkp1}. The formulas also verify directly
\[
 \int_M(2x+ky)x^2=0,\quad
 \int_M(2x+ky)xy=1,\quad
 \int_M(2x+ky)y^2=0,
\]
which are the restrictions to $F$.

The discriminant of the binary cubic $\int_M(sx+ty)^3$ is
\begin{equation}\label{eq:Mkdisc}
 \Delta_k=-27(k^2-1).
\end{equation}
It is unchanged by an integral change of basis or reversal of the ambient orientation. Hence a diffeomorphism forces $|k|=|\ell|$. Conversely, reversing the orientation of $S^2_b$ identifies the normal data with $k$ replaced by $-k$; the corresponding boundary change is absorbed by Lemma~\ref{lem:extension}. This gives $M_k\cong M_{-k}$. Equivalently, the full systems \eqref{eq:Mkcubic}--\eqref{eq:Mkp1} can be compared by Wall--Jupp. Finally \eqref{eq:Mkdef} is a double disk-bundle decomposition with $b_2=2$, so \cite[Theorem A]{DGGK} proves rational ellipticity.
\end{proof}

For $|k|=1$, the cubic has a double root and $M_k$ has the real homotopy type of $S^2\times\CP^2$. For $|k|>1$, its discriminant is negative and its real homotopy type is $\CP^3\#\CP^3$. The nonzero class $w_2=\bar y$ excludes every $M_k$ from the spin models $S^6$, $\CP^3$, $S^2\times S^4$, $\Wflag$, and $\Eflag$. This is not an obstruction to positive curvature by itself.

An explicit representative is
\begin{equation}\label{eq:M1}
 M_1\cong\mathbb P(L^2\oplus\one\oplus\one)\longrightarrow S^2.
\end{equation}
It is induced from the circle action rotating the first homogeneous coordinate of $\CP^2$. The fixed projective line gives the uniform boundary and the remaining coordinate fixed point gives the interior orbit. In a suitable basis its ring is
\[
 \Z[x,y]/(x^2,y^3+2xy^2),\qquad\int xy^2=1.
\]
The quotient of a bi-invariant metric on $G$ and a Fubini--Study metric on $\CP^2$ gives invariant nonnegative curvature. It does not give strict positivity. Thus Figure~2(c) is smoothly realizable, contrary to an interpretation of the diagram as topologically impossible; its positive-curvature realizability remains a separate question.

\subsection{Doubling the quotient excludes two fixed points and an interior orbit}
\begin{lemma}\label{lem:balllinks}
At an isolated $G$-fixed point with tangent representation $\R^3\oplus\R^3$, the quotient link is the round hemisphere $S^2_+(1/2)$. At an interior orbit with slice $(1,1)$ it is $S^2(1/2)$.
\end{lemma}
\begin{proof}
Every invariant inner product on $\R^3\oplus\R^3$ is carried to the standard inner product by an equivariant linear isometry. For the standard one write a representative of $S^5/\SO(3)$ as
\[
 z=e^{i\theta}(\cos t\,e_1+i\sin t\,e_2),\qquad
 0\le t\le\pi/4,\quad\theta\in\R/\pi\Z.
\]
Projecting the coordinate derivatives perpendicular to the orbits gives
\[
 dt^2+\cos^2(2t)\,d\theta^2
 =\tfrac14(d\rho^2+\sin^2\rho\,d\phi^2),
 \qquad\rho=\pi/2-2t,\quad\phi=2\theta.
\]
The dependent-vector locus is the equator. The second assertion is the ordinary unit Hopf quotient; an invariant Hermitian form on the weight-one $\C^2$ slice is likewise equivalent to the standard form.
\end{proof}

\begin{proposition}\label{prop:doublebound}
Let $f$ count the boundary fixed points of the preceding type and let $r$ count interior singular points. Then $f+2r\le3$. In particular case~(e) cannot occur under \eqref{eq:H}.
\end{proposition}
\begin{proof}
Compactness gives $\Sec_g\ge\kappa>0$. Double $X$ along its whole boundary. Petrunin's theorem \cite{Petrunin} gives curvature at least $\kappa$ on $DX$. Each boundary fixed point contributes one point whose link is the double of $S^2_+(1/2)$, namely $S^2(1/2)$. Each interior singular point contributes two points. Their weighted-Hopf links have three-extent at most $\pi/3$, as recalled in Section~\ref{sec:tools}. The four-point extent inequality therefore gives $f+2r\le3$. In~(e), $f=2$ and $r=1$, a contradiction.
\end{proof}
This proof doubles the metric quotient, not the six-manifold across its codimension-two singular set. It also shows that in the finite-fixed-set ball case the existence of two boundary fixed points automatically excludes an interior singular orbit. Lemma~\ref{lem:noextra} then excludes exceptional orbits, and Liu's Theorem~4.4 applies. Its equivariant homeomorphism to the linear six-sphere, together with $\Theta_6=0$, gives the ordinary diffeomorphism assertion for~(d).

The excluded diagram is nevertheless smoothly realizable. Consider
\begin{equation}\label{eq:quadric}
 Q^3=\{[z_1:\cdots:z_5]\in\CP^4:\ z_1^2+\cdots+z_5^2=0\},
\end{equation}
with $\SO(3)$ rotating the first three coordinates. This quadric is the oriented Grassmannian $\SO(5)/(\SO(2)\times\SO(3))$ and is simply connected by the homogeneous-space homotopy sequence. Its fixed points are $[0:0:0:1:i]$ and $[0:0:0:1:-i]$. At $z_4=z_5=0$ the conic is a single circle-isotropy orbit with normal weights $(1,1)$. When $(z_4,z_5)\ne0$, projective isotropy fixes the first three coordinates as a complex vector and is trivial or a circle; there are no exceptional orbits.

To identify the quotient, normalize a vector as $(v,w)\in\C^3\oplus\C^2$. The Gram matrix of the real and imaginary parts of $v$ determines its $\SO(3)$ orbit. The conditions
\[
 \|v\|^2=1-\|w\|^2,\qquad v\cdot v=-w\cdot w
\]
are realizable precisely when $\|w\|^2+|w\cdot w|\le1$. This is an $S^1$-invariant star-shaped four-ball in $\C^2$. Radial rescaling and the scalar circle quotient identify it with $D^4/S^1\cong B^3$. The origin is the interior conic orbit and the boundary has exactly the two fixed points above. Thus Proposition~\ref{prop:doublebound} is an obstruction to positive curvature invariant under this action on $Q^3$, not an assertion that the smooth diagram is impossible.

\subsection{The remaining exceptional ball configuration is $\CP^3$}
Let $I$ be the closed $C_2$ interval in~(f), including its boundary normalizer endpoint and its interior circle endpoint. Put $Q=\pi^{-1}(I)$.

\begin{lemma}\label{lem:conicpiece}
The inverse image $Q$ is a smooth embedded four-manifold equivariantly diffeomorphic to $\CP^2$ with its real $\SO(3)$ action. The interior singular sphere is its complex conic. The oriented normal line has Euler class $\pm h$, where $h$ generates $H^2(\CP^2;\Z)$.
\end{lemma}
\begin{proof}
At the interior endpoint, $Q$ is the weight-two axis plane in the slice $(1,2)$. At the boundary endpoint, it is the standard plane in \eqref{eq:O2slice}. It therefore closes smoothly. Align the principal subgroup as $H=C_2$ about the $z$ axis. At the interior end the singular group is $\SO(2)_z$. At the boundary end it is $\OO(2)_x$, with $x\perp z$: the generator of $H$ lies in the nonidentity component of this normalizer. This relative position, which is not expressed by the abstract symbols $C_2\subset\OO(2)$ alone, is forced by \eqref{eq:O2slice}.

The resulting cohomogeneity-one diagram is the real $\SO(3)$ action on $\CP^2$, with singular orbits $\RP^2$ and the complex conic. The standard cohomogeneity-one classification \cite{Parker,Bredon} gives the equivariant identification. The conic represents twice the projective-line class. Along it the ambient normal plane has circle weight one, and hence Euler number $\pm2$. If $e(\nu Q)=d h$, restriction to the conic gives $2d=\pm2$, so $d=\pm1$.
\end{proof}

\begin{lemma}[The complementary cap]\label{lem:cap}
The complement of an open normal disk bundle of $Q$ is equivariantly diffeomorphic to $D(\R^3\oplus\R^3)$.
\end{lemma}
\begin{proof}
Choose the normal tube compatibly with the boundary-soul product description. Its quotient removes a neighborhood of the exceptional interval and leaves a ball with exactly one fixed point and a circle-isotropy side face. Remove a small linear disk about that fixed point. The remaining quotient is a product $D^2\times[0,1]$ with its side-face structure, and there are no other orbit types in it. The interval vector field lifts to a smooth invariant vector field on the corresponding $G$-manifold: on a face chart $\R^2\oplus\C_1$ choose the lift tangent to the face and use the squared normal radius as quotient normal coordinate. The field has no normal component along that face. On principal charts use a connection, and patch these lifts with an invariant partition of unity. Its flow gives an equivariant product cobordism. Thus the outer boundary is obtained from the boundary of the small linear fixed-point disk only by a collar. The small disk has representation $\R^3\oplus\R^3$, proving the assertion. This is the smooth equivariant collar argument, not an assumption that an arbitrary topological cap has a standard smoothing.
\end{proof}

It follows that
\[
 M\cong D(\mathcal O(1)\to\CP^2)\cup_{S^5}D^6.
\]
The negative Euler choice gives the same unoriented conclusion. The standard attachment is $\CP^3$, decomposed as a tubular neighborhood of a hyperplane and the complementary disk. Changing the boundary attachment can alter this ordinary manifold only by a homotopy six-sphere, which is trivial by \cite{KM}. Hence $M\cong\CP^3$.

The action
\[
 A[z_1:z_2:z_3:z_4]=[A(z_1,z_2,z_3)^t:z_4],\qquad A\in\SO(3),
\]
is Liu's Example~5.3(3) and realizes exactly these strata with the positively curved Fubini--Study metric. The new conclusion is the uniqueness of the underlying smooth manifold, not the existence of that example. This completes Theorem~\ref{thm:mainball}.

% ===== 05_su2_s3.tex =====
\section{\texorpdfstring{$\SU(2)$ actions with orbit space $S^3$}{SU(2) actions with orbit space S3}}\label{sec:su2s3}

Assume $G=\SU(2)$, $X\cong S^3$, and \eqref{eq:H}. By Liu's Proposition~4.9, the singular orbits are $G/T$ spheres, their number is $1\le r\le3$, and their union is exactly $M^{\{\pm1\}}$. They are therefore totally geodesic. Every exceptional isotropy group is cyclic of odd order. At a singular orbit, the normal representation has a pair of coprime nonzero odd weights $(p,q)$, because the central element acts as $-\id$ on the whole normal slice. These are existing results from \cite{LiuCorrected}; the additional curvature restriction is the following theorem.

\subsection{Adjoint-weight matching}
\begin{theorem}\label{thm:matching}
Suppose an open exceptional edge has isotropy $C_k$, $k>1$ odd, and transverse slice $\R\oplus\C_s$. Then
\begin{equation}\label{eq:matching}
 s\equiv\pm2\pmod k.
\end{equation}
Consequently, an exceptional edge incident to a one-valent singular vertex has order three.
\end{theorem}
\begin{proof}
A component of $M^{C_k}$ over the edge is a totally geodesic surface. Its residual rotation-circle action has a horizontal meridian connecting the two singular endpoints. Choose such a unit-speed geodesic $\gamma:[0,L]\to M$, perpendicular to the $G$ orbits. Its interior lies over the open edge. The endpoint isotropy circles can be aligned with the unique maximal torus containing $C_k$.

Let $\mathfrak m=\mathfrak{su}(2)\ominus\mathfrak t$. For $A\in\mathfrak m$, write $J_A=A^*|_\gamma$ for the Killing Jacobi field. Its span $P_t$ is the adjoint weight-two plane. At an interior point, the $C_k$ representation on the tangent space is
\[
 \R^2\oplus\C_2\oplus\C_s.
\]
Suppose $s\not\equiv\pm2\pmod k$. The weight-two isotypic component is then unique, so parallel transport along the pointwise $C_k$-fixed geodesic preserves $P_t$. Let $Q_0$ be a fixed invariant inner product on $\mathfrak m$. Since $k$ is odd and greater than one, the invariant metric on this real rotation representation is a scalar multiple of $Q_0$. Therefore
\begin{equation}\label{eq:Jacobilength}
 \langle J_A,J_B\rangle=f(t)^2Q_0(A,B),\qquad f(t)>0.
\end{equation}
In the identifications $A\mapsto J_A$, the derivative operator on this plane is $a(t)\id+b(t)J$, where $J$ is its complex structure. The Wronskian of orbit Jacobi fields vanishes: the Killing identities express it as a multiple of $\langle\gamma',[A^*,B^*]\rangle$, which is zero because the bracket is orbit tangent and $\gamma'$ is horizontal. Thus $b=0$. Differentiating \eqref{eq:Jacobilength} gives $a=f'/f$, whence
\[
 D_tJ_A=(f'/f)J_A,\qquad D_t^2J_A=(f''/f)J_A.
\]
The Jacobi equation now gives, for every nonzero $A$,
\begin{equation}\label{eq:concavity}
 K(\gamma',J_A)=-f''/f>0.
\end{equation}

At each endpoint the fields $J_A$ span the noncollapsed tangent plane of the singular sphere, so $f$ extends positively and smoothly to the endpoints. That sphere is totally geodesic. The first-variation formula for its orbit metric in the normal direction $\gamma'$ yields
\[
 f'(0)=f'(L)=0.
\]
This contradicts the strict concavity in \eqref{eq:concavity}. The assumption of inequivalent characters is therefore impossible, proving \eqref{eq:matching}.

At a one-valent vertex on an order-$k$ edge the slice weights are $(1,k)$; its transverse character is one. Thus $1\equiv\pm2\pmod k$. With $k>1$ odd, this forces $k=3$.
\end{proof}
The uniqueness of the weight-two isotypic summand is essential. When $s\equiv\pm2$, the two rotation planes can mix under parallel transport, so the preceding scalar concavity argument does not apply. This is precisely the alternative retained by the theorem.

\subsection{All remaining weights}
There is no finite branching vertex, since all exceptional groups are cyclic. Every singular vertex has valency at most two. A one-vertex loop would force equal nontrivial weights and is excluded by coprimality. The remaining abstract possibilities are consequently the no-edge cases, an interval, a path, a bigon, a triangle, and an isolated vertex together with an interval or a bigon, as in the specialization of \cite[Proposition 5.1]{KSR}.

For three singular vertices, a two-edge path has order-three leaves by Theorem~\ref{thm:matching}, forcing middle weights $(3,3)$, which is impossible. In a genuine triangle order the labels $a<b<c$. Along the $c$ edge, slice compatibility gives $a\equiv\pm b\pmod c$. Hence $c=a+b$, contradicting oddness. A genuine bigon together with an isolated singular vertex is excluded by the double branched cover along its cycle: \cite[Proposition 2.14]{KSR} preserves positive Alexandrov curvature, and its small-point conclusion gives the two vertices on the cycle and two lifts of the isolated vertex. Four such points contradict the extent bound. Thus the only three-vertex patterns are three isolated vertices or one interval together with an isolated vertex. In the second case its order is three.

For two singular vertices, the unordered slice pair $(p,q)$ is the same at both endpoints, each order greater than one being recorded by its edge. If $p=q$, coprimality gives $(1,1)$. Otherwise take $1\le p<q$. The order-$q$ edge is genuine and has transverse character $p$. By \eqref{eq:matching},
\[
 p\equiv\pm2\pmod q.
\]
Since $p$ is odd and $0<p<q$, $p=q-2$. Thus the necessary list is
\begin{equation}\label{eq:SU2pairs}
 (p,q)=(1,1)\quad\text{or}\quad(p,q)=(2j-1,2j+1),\qquad j\ge1.
\end{equation}
These pairs satisfy the matching congruences on both genuine edges, but that converse arithmetic statement is not a converse curvature theorem.

\begin{table}[tbp]
\centering\small
\begin{tabularx}{\textwidth}{@{}lXX@{}}
\toprule
Figure~4 pattern & Remaining labels & Smooth action\\
\midrule
Two vertices & $(1,1)$ or $(2j-1,2j+1)$, $j\ge1$ & $N^G_{1,1}$ or $N^G_{2j-1,2j+1}$ on $S^2\times S^4$.\\
Bigon and isolated vertex & $(1,1)$ or $(1,3)$ & Block action on $\Wflag$ or commuting right block action on $\Eflag$.\\
Three-sided cycle & $(1,1,1)$ or $(1,1,3)$ & The same two flag actions after deleting labels one.\\
\bottomrule
\end{tabularx}
\caption{The label conventions of Liu's Figure~4, interpreted with order-one edges erased. The two-vertex row is a necessary positive-curvature candidate list; both three-vertex rows have the stated positively curved realizations.}\label{tab:SU2labels}
\end{table}

\subsection{Two singular orbits: reconstruction of the weighted product actions}
For a two-vertex bigon, the cycle is unknotted by the argument for graph~8 in \cite[Section 6]{KSR}. The interval and no-edge cases need no knot obstruction. Removing the two singular tubes therefore gives the orbifold product $S^2(p,q)\times[0,1]$. Lifting the interval direction and applying Lemma~\ref{lem:extension} gives
\begin{equation}\label{eq:SU2double}
 M\cong_G U^G_{p,q}\cup U^G_{p,q}
 =G\times_TS(\R\oplus\C_p\oplus\C_q)=N^G_{p,q}.
\end{equation}
For odd $p,q$, the oriented five-plane bundle over $G/T=S^2$ has $w_2=p+q\bmod2=0$, hence is trivial. The underlying manifold is $S^2\times S^4$.

Conversely, every positive coprime odd pair gives a smooth effective induced action with two singular orbits and the indicated slice pair. The diagonal quotient of a bi-invariant metric on $G$ and a round metric on the four-sphere gives invariant nonnegative curvature. In particular all pairs in \eqref{eq:SU2pairs} are smoothly realizable. This construction does not resolve invariant positive curvature, including for $N^G_{1,1}$, which is the diagonal action in Liu's Remark~4.12.

\subsection{Three singular orbits with one exceptional interval}
\begin{proposition}\label{prop:SU2Eunique}
A $\SU(2)$ action satisfying \eqref{eq:H} with three singular orbits and an exceptional interval is equivariantly diffeomorphic to the commuting right block action on the Eschenburg flag $\Eflag$.
\end{proposition}
\begin{proof}
The interval has order three. Its inverse image is
\[
 A_3=G\times_TS(\R\oplus\C_3)=\mathbb P(\one\oplus L^3)\to S^2,
\]
and the third singular orbit has slice $(1,1)$. The absolute normal characters at the two sections of $A_3$ are both one. Orient the normal line so that the first character is $u=1$. Compatibility in Lemma~\ref{lem:normal} forces the other character to be $v=1$, because $-1\not\equiv1\pmod3$. Consequently its Euler class is one on a section and zero on a fiber. The line is equivariantly
\begin{equation}\label{eq:Eschnormal}
 \nu A_3\cong\rho^*L,
\end{equation}
where $\rho:A_3\to S^2$ is the bundle projection. In induced coordinates its disk bundle is
\[
 G\times_T\bigl(S(\R\oplus\C_3)\times D(\C_1)\bigr).
\]
The unit normal coordinate has weight one, so fixing its phase identifies the boundary equivariantly with $G\times S^2$.

Remove this edge neighborhood and a tube about the remaining singular orbit. The quotient complement is a free $S^2\times[0,1]$ product, and the principal $\SU(2)$ bundle on it is trivial. Absorb it into a collar. The complete manifold is therefore the union of the disk bundle in \eqref{eq:Eschnormal} and $U^G_{1,1}$. Lemma~\ref{lem:extension} absorbs every attaching map on the free boundary. The endpoint characters determine the normal line, so there is a unique equivariant diffeomorphism type with these data. The explicit action verified below realizes this type on $\Eflag$.
\end{proof}
This is a smooth reconstruction, not merely the agreement of Betti numbers recorded in Liu's Example~5.5. Replacing $3$ by any odd $k\ge3$ in the same disk-bundle construction still gives a smooth action with one $C_k$ interval. Theorem~\ref{thm:matching} excludes its positive-curvature realization for $k>3$. Thus the restriction to order three is geometric, not a local smooth-slice restriction.

\subsection{The Eschenburg model and its stabilizers}
Use the determinant-preserving free torus action
\begin{equation}\label{eq:Etorus}
 (z,w)\cdot g=\diag(z,w,zw)\,g\,\diag(1,1,z^2w^2)^{-1},
 \qquad g\in\SU(3).
\end{equation}
Its quotient is $\Eflag$. The two diagonal factors have the same determinant, so \eqref{eq:Etorus} preserves $\SU(3)$ even though neither factor separately need belong to it. The right upper-left $\SU(2)$ block commutes with this torus action. To express the block action as a left action, write $A\cdot[g]=[g\diag(A,1)^{-1}]$.

For completeness, the stabilizer computation needed to identify the reconstruction is short. If a torus element fixes $g$, the multisets $\{z,w,zw\}$ and $\{1,1,(zw)^2\}$ agree. Two of $z,w,zw$ must then be one, forcing $z=w=1$. Thus the torus action is free. If $v=(v_1,v_2,v_3)$ is the last column of $g$, an element of the torus contributes to the block stabilizer precisely when
\[
 \diag(z,w,zw)v=(zw)^2v.
\]
For the nonzero coordinates the respective conditions are
\[
 v_1\ne0:\ zw^2=1,\qquad
 v_2\ne0:\ z^2w=1,\qquad
 v_3\ne0:\ zw=1.
\]
All three nonzero coordinates give trivial stabilizer. A single nonzero coordinate gives a circle, accounting for three singular orbits. Supports $\{1,3\}$ and $\{2,3\}$ also give trivial stabilizer. On support $\{1,2\}$, the equations yield $z=w$ and $z^3=1$, giving the noncentral order-three subgroup in the block. The parameter $|v_1|^2$ runs along its open quotient interval. The third circle orbit is isolated from that interval. This is exactly the pattern in Proposition~\ref{prop:SU2Eunique}.

Positive curvature for this model and invariance under the commuting block action are established in the Eschenburg construction and its symmetry analysis \cite{Eschenburg,EschenburgInhom,GSZ,LiuCorrected}. We use that metric theorem directly, without reproducing its curvature computation. The ordinary flag $\Wflag$ with the left upper-left block action likewise has three circle orbits and no exceptions and carries its invariant Wallach metric \cite{Wallach,LiuCorrected}.

The integral data for these flag manifolds are standard and can be obtained from biquotient transgression and the tangent representation; see \cite{EschenburgCohom,GKZsymplectic,GKZ}. In a common convenient convention they are
\begin{align}
 H^*(\Wflag;\Z)&=\Z[x,y]/(x^3,y^2-xy+x^2),& p_1(\Wflag)&=0,\label{eq:Wflag}\\
 H^*(\Eflag;\Z)&=\Z[x,y]/(x^3,y^2-xy-x^2),& p_1(\Eflag)&=8x^2.\label{eq:Eflag}
\end{align}
Both have torsion-free cohomology, are spin, and are oriented by $\int x^2y=1$. To relate \eqref{eq:Eflag} to the torus characters $a,b$ in \eqref{eq:Etorus}, take $x=a+b$, $y=b$: the familiar relations are $a^2+3ab+b^2=0$ and $ab(a+b)=0$. For \eqref{eq:Wflag}, the usual relations are $a^2+ab+b^2=0$, $ab(a+b)=0$, and one may take $x=a$, $y=-b$. These changes of generators reconcile the two common sign conventions; they do not change the characteristic classes.

In particular the cubic polynomials in this convention are
\[
 \int_{\Wflag}(sx+ty)^3=3s^2t+3st^2,\qquad
 \int_{\Eflag}(sx+ty)^3=3s^2t+3st^2+2t^3.
\]
Their real root types differ. More decisively for smooth recognition, the first Pontryagin class is zero on $\Wflag$ and nonzero on $\Eflag$. Equal Betti numbers do not identify these manifolds.

\subsection{Completion of the $\SU(2)$ classification}
For three singular orbits and no exceptions, Liu's corrected Theorem~4.11 already gives the equivariant action on $\Wflag$. For one singular orbit, the absence of a one-vertex loop implies there are no exceptions, so the same theorem gives the linear action
\begin{equation}\label{eq:SU2sphere}
 A(u,v)=(Au,\Ad_Av),\qquad
 (u,v)\in S(\C^2\oplus\mathfrak{su}(2)).
\end{equation}
There is no need to repeat Liu's exceptional-free gluing proof. This representation has one singular sphere and no $G$-fixed point. It should be distinguished from $\C^2\oplus\R^3_{\mathrm{triv}}$, the triple suspension of the Hopf action, which has a fixed $S^2$ and a ball quotient. Thus the explicit model in Liu's Example~5.2(1), rather than a literal suspension description, is used in the spherical-quotient case.

The weight reductions, \eqref{eq:SU2double}, and Proposition~\ref{prop:SU2Eunique} prove Theorem~\ref{thm:mainSU2}. The remaining uncertainty in its two-orbit row is metric realizability on a determined smooth action, not a missing diffeomorphism classification.

% ===== 06_reconstruction.tex =====
\section{Global consequences and the remaining reconstruction problem}\label{sec:reconstruction}

\subsection{Rational ellipticity throughout the symmetry class}
\begin{proof}[Proof of Corollary~\ref{cor:globalelliptic}]
Liu's orbit-space dichotomy gives $X\cong S^3$, $B^3$, or $B^4$. If $X\cong S^3$, rational ellipticity is Theorem~A of Khalili Samani--Radeschi \cite{KSR}. The $B^4$ case and the $\SU(2)$ cases with boundary are identified by Liu with $S^6$ or $\CP^3$ \cite[Theorems~1.1--1.2 and~3.1]{LiuCorrected}.

For $G=\SO(3)$ and $X\cong B^3$, Section~\ref{sec:ball} exhausts the configurations. Cases~(a), (d), and~(f) give a sphere or $\CP^3$, and case~(e) is excluded. In case~(b), the fibration $S^2\to E_n\to S^4$ implies rational ellipticity by the rational homotopy exact sequence. In case~(c), the decomposition \eqref{eq:Mkdef} is a double disk bundle with $b_2(M_k)=2$. Theorem~A of DeVito--Galaz-Garc\'ia--Kerin \cite{DGGK} applies: its six-dimensional hypothesis is $b_2\le3$, which is satisfied here. Thus all cases are rationally elliptic.

Liu gives $\chi(M)\in\{2,4,6\}$. A simply connected rationally elliptic space with positive Euler characteristic has no odd-degree rational cohomology \cite{FHT}. Poincar\'e duality now gives $\chi(M)=2+2b_2(M)$. Herrmann's classification of rationally elliptic six-manifolds \cite{Herrmann}, restricted to this Betti range and positive Euler characteristic, yields the six real homotopy types stated in the corollary.
\end{proof}

The conclusion combines two independent geometric mechanisms: the boundaryless quotient theorem and the explicit disk-bundle reconstructions for quotients with boundary. No rational-ellipticity hypothesis is imposed on $M$ in addition to \eqref{eq:H}. On the other hand, rational ellipticity supplies no general torsion-freeness theorem. In the cyclic and ball cases, torsion-freeness follows from the actual decompositions; it is not inferred from Corollary~\ref{cor:globalelliptic}.

\subsection{What is recorded by a graph}
The tables in Section~\ref{sec:so3s3} record an \emph{abstract weighted graph}. They distinguish singular vertices, branching exceptional vertices, and cyclic edge orders. They do not by themselves specify a smooth $G$-manifold. Standard equivariant reconstruction \cite[Chapters~IV--V]{Bredon} requires the following additional information.

First one specifies a tame embedding of the graph into $S^3$, together with the actual isotropy subgroups, not only their abstract isomorphism types. At each vertex one specifies the slice representation and the inclusions of the edge groups into the vertex group. The normalizer transition data along an edge must intertwine these representations. These conditions include the transverse-character congruences of Lemma~\ref{lem:edgecong} and the normal Euler-class constraints of Lemma~\ref{lem:normal}.

Let $B$ be the closure of the complement of a small stratified neighborhood of the graph. Over $B$ one specifies a principal $G$-bundle and its identifications with the boundary bundles of the local models. Principal $\SO(3)$-bundles over a three-dimensional CW complex are classified by $w_2\in H^2(B;\Z/2)$; principal $\SU(2)$-bundles over such a complex are trivial \cite{Husemoller,MS}. Finally one specifies the equivariant attaching maps, modulo automorphisms extending over the pieces. In the cyclic configurations, the unknotting input and Lemma~\ref{lem:extension} remove these apparent extra choices. That reduction has not been established for every branching graph.

This distinction also explains why the present graphs should not be confused with GKM graphs. Under the hypotheses of GKM theory, a higher-rank torus action can encode integral equivariant characteristic classes in a labelled graph; the six-dimensional recognition theorem of Goertsches--Konstantis--Zoller is a particularly relevant example \cite{GKZ}. An abstract finite-isotropy graph for a rank-one nonabelian action does not automatically satisfy those hypotheses or retain the same information. Its edge orders alone do not determine an integral degree-two lattice.

\subsection{Integral completion of the branching cases}
For a specified reconstruction, Mayer--Vietoris determines the integral homology from the inclusion maps of the boundary pieces. These maps, rather than just the ranks of the groups, decide torsion and the primitive lattice
\[
 H^2(M;\Z)\hookrightarrow\bigoplus_i H^2(S_i;\Z).
\]
The rational image is the hyperplane in \eqref{eq:lin}. In a rank-one lollipop case, the integer $t$ in \eqref{eq:lollinumbers} measures information that this rational hyperplane does not detect. Once the lattice is known, Proposition~\ref{prop:localization} gives the cubic cup form and the Pontryagin covector. The transition functions of the tangent bundle, or an equivalent integral normal-bundle calculation, are still needed to determine $w_2$.

When homology is torsion-free, the Wall--Jupp theorem \cite{Wall,Jupp} converts the resulting tuple \eqref{eq:WJ} into an ordinary diffeomorphism classification. If torsion occurs, one must retain its additional linking and normal data and use the appropriate classification framework, rather than apply the torsion-free theorem \cite{Zubr}. In particular, rational cohomology of a sphere does not imply that a theta-graph manifold is a homotopy sphere. If torsion-freeness is separately established, simple connectivity and the Hurewicz theorem do give a homotopy sphere, and $\Theta_6=0$ identifies it with $S^6$ \cite{KM}.

For comparison, the standard characteristic data of several recognition models are recorded below. The degree-two generators and the orientation are part of each entry. The flag data are given in \eqref{eq:Wflag}--\eqref{eq:Eflag}; the data for $E_n$, $P_{a,b}$, and $M_k$ have already been computed.
\begin{center}
\begin{tabular}{p{31mm}p{61mm}p{48mm}}
\toprule
Model & Integral intersection data & Characteristic classes\\
\midrule
$S^6$ & $H^2=H^4=0$ & $w_2=p_1=0$\\
$S^2\times S^4$ & $x^2=0$, $\ang{xy,[M]}=1$ & $w_2=p_1=0$\\
$\CP^3$ & $\ang{x^3,[M]}=1$ & $w_2=0$, $p_1=4x^2$\\
$S^2\times\CP^2$ & $b^2=a^3=0$, $\ang{a^2b,[M]}=1$ & $w_2=\bar a$, $p_1=3a^2$\\
$\CP^3\#\CP^3$ & $xy=0$, $\ang{x^3,[M]}=\ang{y^3,[M]}=1$ & $w_2=0$, $p_1=4x^2+4y^2$\\
\bottomrule
\end{tabular}
\end{center}
The product and projective-space formulas follow from the standard tangent-bundle identities \cite{MS}; connected-sum characteristic numbers are additive. These are recognition data, not assertions of positive-curvature existence for all entries.

\subsection{Overlaps and limits of the classification}
Different orbit stratifications can occur on the same smooth manifold. The family $E_n$ appears in the uniform-boundary ball case, while Proposition~\ref{prop:dihedralfamily} realizes $E_{(k^2-1)/8}$ with a dihedral lollipop for every odd $k\ge3$. Likewise, $M_1$ is the nontrivial unoriented linear $\CP^2$-bundle over $S^2$; it is diffeomorphic to $P_{1,2}$, although the displayed actions have different orbit-space topology. The assertion follows from their clutching classes modulo three, or from their integral characteristic data. These overlaps do not identify the actions equivariantly.

A complete classification under \eqref{eq:H} would require two further steps. In the explicitly reconstructed families, one must determine which parameters admit invariant positive curvature. In the remaining $\SO(3)$ branching families, one must first finish integral reconstruction and then address curvature. Neither task follows merely from the finite graph catalogue or the bound $b_2\le2$.

The frequently suggested finite target
\begin{equation}\label{eq:target}
 S^6,\qquad\CP^3,\qquad S^2\times S^4,\qquad\Wflag,\qquad\Eflag
\end{equation}
is therefore a classification target, not the conclusion of this paper. For a torsion-free candidate, agreement of the Wall--Jupp data with an entry in \eqref{eq:target} is an exact smooth-recognition criterion. The missing geometric work is to force that agreement and to decide the remaining invariant metric questions. The next section addresses part of the latter issue without confusing a restricted curvature obstruction with a complete classification.

% ===== 07_metrics.tex =====
\section{Curvature obstructions for the diagonal product actions}\label{sec:metrics}

The preceding classifications leave several actions on $S^2\times S^4$ as smooth candidates. We now isolate what can be proved about the two diagonal actions discussed by Liu. The restrictions below do not require or imply a positive-curvature metric on any of the other smooth models.

\subsection{The two actions and the space of invariant metrics}
For the $\SO(3)$ action, write
\begin{equation}\label{eq:SO3productaction}
 M=S^2\times S(\R^3\oplus\R^2),\qquad
 A(x,y,z)=(Ax,Ay,z).
\end{equation}
The orbit map and singular set are
\begin{equation}\label{eq:SO3productquotient}
 q(x,y,z)=(x\cdot y,z)\in B^3,\qquad
 F=\{(x,tx,z):t^2+|z|^2=1\}\cong S^2\times S^2.
\end{equation}
Indeed, the Gram data $|x|=1$, $|y|^2=1-|z|^2$, and $x\cdot y=t$ classify pairs of vectors up to $\SO(3)$. Isotropy is trivial when $x,y$ are independent and is $\SO(2)$ when they are dependent. This is the $n=0$ member of case~(b). Its normal bundle is $\nu F\cong\operatorname{pr}_1^*TS^2$, with Euler class $2a$.

For the $\SU(2)=\Sp(1)$ action, quaternionic coordinates give
\begin{equation}\label{eq:SU2productaction}
 M=S(\operatorname{Im}\HH)\times S(\HH\oplus\R),\qquad
 a(u,q,t)=(aua^{-1},aq,t).
\end{equation}
The two singular orbits are $q=0$, $t=\pm1$. They are exactly the fixed components of the center. For $q\ne0$, the quotient coordinates are
\[
 \left(\frac{\bar q u q}{|q|^2},t\right)\in S^2\times(-1,1).
\]
Collapsing each end sphere gives $S^3$. This is $N^{\SU(2)}_{1,1}$. The center acts nontrivially on the quaternionic coordinate, so this action does not descend to the $\SO(3)$ action in \eqref{eq:SO3productaction}.

Both actions have an induced description $G\times_T S^4$. A connection on $G\to G/T$ identifies their tangent bundle with
\[
 G\times_T\bigl((S^4\times\mathfrak m)\oplus TS^4\bigr),
 \qquad\mathfrak m=\mathfrak g\ominus\mathfrak t.
\]
Consequently an arbitrary invariant metric is exactly a smooth $T$-equivariant positive-definite block form
\begin{equation}\label{eq:allmetrics}
 \begin{split}
 g_z\bigl((A,U),(B,V)\bigr)
 ={}&a_z(A,B)+c_z(A,V)+c_z(B,U)+b_z(U,V),
 \end{split}
\end{equation}
where $A,B\in\mathfrak m$ and $U,V\in T_zS^4$. This is a parametrization of all invariant metrics, not a connection-metric ansatz. In particular the cross term $c_z$ cannot in general be discarded. At a pole of the $\SU(2)$ slice, equivariance forces this cross term to vanish because the orbit tangent has weight two and the normal slice has weight one. It does not force the cross term to vanish away from the poles.

Nor does invariance force either product projection to be a Riemannian submersion with totally geodesic fibers. For example, $(2+t)^2g_{S^2}+g_{S^4}$ is invariant under \eqref{eq:SU2productaction} but has varying orbit radii. The invariant function $\langle\bar q u q,i\rangle$ can be used in a positive conformal factor to break the full commuting right-$\Sp(1)$ symmetry. In \eqref{eq:SO3productaction}, the positive invariant conformal factor $2+z_1$ breaks the standard suspension-coordinate $\OO(2)$ symmetry. These examples are not positive-curvature constructions; they show that the additional hypotheses in the obstructions below are genuine restrictions.

\subsection{Cheeger deformations and conformal changes}
\begin{proposition}\label{prop:productCheeger}
For either action, let $g_0$ be a $G$-invariant product metric on $S^2\times S^4$. Every diagonal Cheeger deformation $g_s$, $s\ge0$, has a zero-curvature two-plane at every point. No conformal rescaling of any $g_s$ has everywhere positive sectional curvature.
\end{proposition}
\begin{proof}
Use the Cheeger submersion \cite{CheegerNotes}
\[
 \sigma:(M\times G,g_0+s^{-1}Q)\longrightarrow(M,g_s),\qquad
 \sigma(p,a)=a^{-1}p,
\]
where $Q$ is bi-invariant and $s>0$. At $p$ choose $0\ne X\in T S^2$ and $0\ne Y\in T S^4$ orthogonal to the $G$-orbit in the second factor. Such $Y$ exists because those orbits have dimension at most three. Define $L\in\mathfrak g$ by $Q(L,Z)=g_0(X,Z^*)$. At $(p,1)$ the vectors
\[
 \widehat X=(X,-sL),\qquad\widehat Y=(Y,0)
\]
are horizontal. Their ambient curvature numerator is zero: the $M$ components form a mixed product plane and the second $G$ component vanishes.

The O'Neill term also vanishes. The metric-dual one-form $\eta_Z$ of a vertical fundamental field is a sum of forms from $M$ and $G$. Its $M$ part is the sum of one-forms pulled back from the two product factors. Hence $d\eta_Z(\widehat X,\widehat Y)=0$ for every $Z$. For horizontal extensions this implies $[\widehat X,\widehat Y]^{\mathrm{vert}}=0$. O'Neill's horizontal curvature formula \cite{ONeill} therefore gives a zero plane in $(M,g_s)$. The case $s=0$ is immediate.

More generally, suppose a metric $g$ on a compact manifold has a nonpositive plane at every point. At a minimum of a smooth function $f$, its gradient vanishes and its Hessian is nonnegative. The conformal curvature formula \cite{Petersen} gives, for a $g$-orthonormal basis $U,V$ of such a plane,
\[
 K_{e^{2f}g}(U\wedge V)
 =e^{-2f}\bigl(K_g(U\wedge V)-\Hess_gf(U,U)-\Hess_gf(V,V)\bigr)\le0.
\]
Apply this to each $g_s$.
\end{proof}

The same minimum argument excludes a warped product over a closed base with nontrivial spherical fiber: the mixed curvature is $-f^{-1}\Hess f$ at a minimum of the warping function. It does not apply to the full block form \eqref{eq:allmetrics}.

\subsection{Both totally geodesic product fibrations are excluded}
\begin{theorem}\label{thm:noTGproduct}
There is no positively curved metric on $S^2\times S^4$ for which either standard product projection is a Riemannian submersion with totally geodesic fibers. No invariance assumption is needed.
\end{theorem}
\begin{proof}
We use O'Neill's mixed curvature formula for totally geodesic fibers \cite{ONeill}:
\begin{equation}\label{eq:ONmixed}
 K(X,V)=|A_XV|^2
\end{equation}
for horizontal and vertical unit vectors. For the projection onto $S^2$, the linear map $A_X$ sends the four-dimensional vertical space into the one-dimensional horizontal space orthogonal to $X$. It has a nonzero kernel, producing a zero mixed plane.

Consider the projection onto $S^4$. We first show that the two-sphere fibers would necessarily be round; this is not an initial assumption. Totally geodesic fibers are mutually isometric under horizontal transport. The curvature of this transport connection is a vertical Killing field on each fiber, represented, up to the conventional sign and factor, by $2A_XY$. At a point of a fiber, if the values of these Killing fields failed to span its tangent plane, a nonzero vertical $V$ would be perpendicular to all $A_XY$. The identity $\langle A_XV,Y\rangle=-\langle V,A_XY\rangle$ would imply $A_XV=0$ for all $X$, contradicting \eqref{eq:ONmixed}. The fiber isometry group therefore acts transitively. Its Gaussian curvature is constant, and Gauss--Bonnet makes that constant positive. All fibers are isometric round spheres of one radius.

The connection thus reduces to an $\SO(3)$ connection on the round sphere bundle. Let $\Omega$ be its curvature, identified fiberwise with a two-form on the base valued in $\R^3$. At a unit sphere point $p$, the horizontal O'Neill tensor is
\[
 A_XY=-\tfrac12\Omega(X,Y)\times p.
\]
For every nonzero $w\in\R^3$, the two-form $\Omega_w=\langle\Omega,w\rangle$ on the four-dimensional base must be nondegenerate. Otherwise choose $X\ne0$ in its kernel, $p\perp w$, and the vertical vector $V=w\times p$. Since $p\times V=w$, the displayed formula implies $A_XV=0$, again contradicting positivity. This is the fatness condition for the associated two-sphere bundle.

The standard coupling-form construction for a fat round two-sphere bundle now gives a symplectic form on its total space; see Fine--Panov \cite[Section~2.2]{FinePanov}. We use that theorem directly. In a trivialization its familiar local expression is the fiber area form minus $d\langle p,\gamma\rangle$, with $\gamma$ the connection one-form, and nondegeneracy is precisely the condition on $\Omega_w$ above. But
\[
 H^2(S^2\times S^4;\R)=\R\alpha,\qquad\alpha^2=0.
\]
Every degree-two class has zero cube, whereas a symplectic form on a closed six-manifold has nonzero cubic integral. This contradiction proves the theorem.
\end{proof}

The argument extracts fatness from the submersion hypothesis. The absence of a symplectic form by itself is not an obstruction to positive sectional curvature: $S^6$ is the simplest warning against that inference.

\subsection{Additional symmetries forbidden by positive curvature}
\begin{proposition}\label{prop:additional}
For the action \eqref{eq:SO3productaction}, a positively curved invariant metric cannot extend to the diagonal $\OO(3)$ action. Moreover no nonidentity standard transformation
\[
 (x,y,z)\longmapsto(x,y,Rz),\qquad R\in\OO(2),
\]
can be an isometry of a positively curved metric on $M$.
For \eqref{eq:SU2productaction}, a positively curved invariant metric cannot also be invariant under the full commuting right-$\Sp(1)$ action on $q$.
\end{proposition}
\begin{proof}
Fixed components of isometries are totally geodesic. We apply Wilking's connectedness theorem \cite{WilkingConnected}: a closed codimension-$k$ totally geodesic submanifold of a positively curved $n$-manifold has $(n-2k+1)$-connected inclusion.

A reflection in the diagonal $\OO(3)$ action fixes $S^1\times S^3$, of codimension two. Its inclusion would be three-connected, contradicting its nontrivial fundamental group. A nonidentity rotation of the suspension coordinates fixes $S^2\times S^2$, also of codimension two. Its second homotopy group has rank two, whereas $\pi_2(M)$ has rank one. A reflection of the suspension coordinates fixes $S^2\times S^3$, of codimension one. Five-connectedness would identify its $\pi_3\cong\Z^2$ with $\pi_3(M)\cong\Z$, again impossible.

For the quaternionic action, adjoining right-$\Sp(1)$ makes the action cohomogeneity one: the regular orbits are the level sets of $t$. Verdiani's classification of positively curved, even-dimensional, simply connected cohomogeneity-one manifolds \cite{Verdiani}, also discussed in \cite{GWZ}, would make $M$ a compact rank-one symmetric space. The integral cohomology ring of $S^2\times S^4$ excludes this possibility.
\end{proof}

\subsection{An action-specific boundary identity for the $\SO(3)$ action}
The union $F$ of singular orbits in \eqref{eq:SO3productquotient} is not the fixed set of a single subgroup. It need not be totally geodesic. The following calculation keeps exactly the second-fundamental-form term that this distinction permits.

Write $F=S^2\times B$, $B=S^2$, with orbit tangent bundle $\mathcal V$ and base tangent bundle $TB$. The isotropy $T$ has weights one on $\mathcal V$ and $\nu F$, and weight zero on $TB$. Thus every invariant metric induces
\begin{equation}\label{eq:boundarywarped}
 g|_F=f(b)^2g_{S^2(1)}+h,
 \qquad f:B\to(0,\infty),
\end{equation}
where $f$ is the intrinsic singular-orbit radius. The second fundamental form satisfies
\begin{equation}\label{eq:IIblocks}
 \mathrm{II}(\mathcal V,\mathcal V)=0,
 \qquad \mathrm{II}(TB,TB)=0.
\end{equation}
Indeed, the symmetric square of weight one contains only weights zero and two, not weight one; the zero-weight summand also cannot map equivariantly to the normal plane.

Along the section $\{e_3\}\times B$, choose unit orbit and normal frames; they exist globally because both plane bundles restrict trivially there. Using the complex structures defined by the circle action, the remaining block is encoded by
\begin{equation}\label{eq:boundarybeta}
 \beta\in\Omega^1(B;\C),\qquad
 |\mathrm{II}(U,X)|^2=|\beta(X)|^2
 \quad (|U|=1,\ U\in\mathcal V).
\end{equation}
Equivariance extends this identity to every orbit. A change of unit normal frame multiplies $\beta$ by a unit complex function, leaving its real rank and $\beta^*\beta$ unchanged.

\begin{proposition}\label{prop:boundaryidentity}
For orthonormal orbit vectors $U,V$ and $h$-orthonormal base vectors $X,Y$, the ambient sectional curvatures along $F$ satisfy
\begin{align}
 K_g(U,V)&=\frac{1-|df|_h^2}{f^2},\label{eq:boundaryvertical}\\
 K_g(X,Y)&=K_h,\label{eq:boundarybase}\\
 K_g(U,X)&=-\frac{\Hess_hf(X,X)}{f}+|\beta(X)|^2.\label{eq:boundarymixed}
\end{align}
These identities concern the indicated planes, not all planes in $TM|_F$.
\end{proposition}
\begin{proof}
Apply the standard warped-product curvature formulas to \eqref{eq:boundarywarped}, and then the Gauss equation \cite{Petersen}. The pure blocks of the second fundamental form vanish by \eqref{eq:IIblocks}. The remaining Gauss correction on a mixed plane is $+|\mathrm{II}(U,X)|^2$, which is \eqref{eq:boundarybeta}.
\end{proof}

\begin{corollary}\label{cor:boundaryobstruction}
If $g$ is invariant under \eqref{eq:SO3productaction} and $\Sec_g>0$, then
\begin{equation}\label{eq:boundarynecessary}
 K_h>0,\qquad |df|_h<1,\qquad
 \mathcal S:=\beta^*\beta-f^{-1}\Hess_hf>0.
\end{equation}
The form $\beta_b:T_bB\to\C$ is a real-linear isomorphism at every minimum of $f$, and at every degeneracy point it satisfies
\[
 \Hess_hf(X,X)<0\quad\text{for all }0\ne X\in\ker\beta_b.
\]
In particular no such positive metric has constant singular-orbit radius, or has $\operatorname{rank}_{\R}\beta\le1$ everywhere. In addition,
\begin{equation}\label{eq:boundaryenergy}
 \int_B|\beta|_h^2\,dA_h>
 \int_B|d\log f|_h^2\,dA_h.
\end{equation}
If $\Sec_g\ge\kappa>0$, the right-hand side can be increased by $2\kappa\operatorname{Area}(B,h)$, with $\ge$ in place of $>$.
\end{corollary}
\begin{proof}
The pointwise inequalities follow directly from Proposition~\ref{prop:boundaryidentity}. At a minimum the Hessian is nonnegative, so $\mathcal S>0$ forces injectivity of $\beta$. A globally injective $\beta:TS^2\to S^2\times\C$ would trivialize $TS^2$, contradicting its Euler number two. Hence $\beta$ degenerates somewhere, ruling out constant $f$. Tracing $\mathcal S>0$ and integrating gives \eqref{eq:boundaryenergy}, since
\[
 \int_B f^{-1}\Delta_h f\,dA_h=
 \int_B f^{-2}|df|_h^2\,dA_h.
\]
The quantitative assertion follows by tracing the lower bound $\mathcal S\ge\kappa h$.
\end{proof}

Even the unit round product metric has a nonzero mixed second fundamental form on $F$. Its induced metric is $(1+t^2)g_{S^2(1)}+g_{B(1)}$. At $t=0$, $x=e_3$, take the orbit tangent $U=(e_1,0,0)$ and the base tangent $X=(0,e_3,0)$ corresponding to increasing $t$. Differentiation of \eqref{eq:SO3productquotient} gives $\mathrm{II}(U,X)=(0,e_1,0)\ne0$. In \eqref{eq:boundarymixed}, the Hessian term $-1$ and the normal term $+1$ cancel. Discarding the latter would give an incorrect curvature calculation.

\subsection{Compatibility of the boundary inequalities}
\begin{proposition}\label{prop:boundarycompatible}
There are smooth data $(h,f,\beta)$ on $B=S^2$ satisfying all the inequalities in \eqref{eq:boundarynecessary}. They occur as the induced metric and mixed second fundamental form along $F$ for a smooth invariant metric on $M$. This assertion does not give positive curvature on all ambient planes.
\end{proposition}
\begin{proof}
Write $b=(b_1,b_2,r)\in S^2$, $\zeta=b_1+ib_2$, and choose $0<\eps\le\tfrac12$. Set
\begin{equation}\label{eq:boundarydata}
 h=g_{S^2(1)},\qquad f=1-\eps r,\qquad
 \beta=\sqrt\eps\bigl((1+r)d\zeta-\zeta\,dr\bigr).
\end{equation}
These are globally smooth. Direct calculation gives
\[
 \beta^*\beta=\eps(1+r)^2h,
 \qquad \Hess_hf=\eps r h,
 \qquad |df|_h\le\eps<1.
\]
Thus
\[
 \mathcal S=\eps\left((1+r)^2-\frac{r}{1-\eps r}\right)h>0.
\]
For $r<0$ both contributions are nonnegative and their sum is positive. For $0\le r\le1$, the expression in parentheses is at least $(1+r)^2-2r=1+r^2$. Also $K_h=1$. The degeneracy of $\beta$ is at the maximum of $f$, not at its minimum.

To realize these tensors, equip $E=\operatorname{pr}_1^*TS^2\to F$ with its standard fiber metric and an invariant connection. Prescribe $g_F=f^2g_{S^2}+h$ and let $\mathrm{II}$ have the pure blocks \eqref{eq:IIblocks} and, for an orbit coordinate vector $u\in T_xS^2$, mixed block
\[
 \mathrm{II}(u,X)=f\bigl(\operatorname{Re}\beta(X)u+
               \operatorname{Im}\beta(X)J_xu\bigr)\in E_{(x,b)}.
\]
In an equivariant tubular neighborhood, set for small $n\in E_p$
\[
 g_n(A^h,B^h)=g_F(A,B)-2\langle\mathrm{II}(A,B),n\rangle,
\]
use the Euclidean metric in the fiber directions, and make horizontal and vertical directions orthogonal. Compactness gives a uniform neighborhood on which this tensor is positive definite. Its normal derivative at the zero section realizes the prescribed second fundamental form. An invariant cutoff and a convex interpolation with any invariant metric extend it to all of $M$ without changing it near $F$.

The planes in Proposition~\ref{prop:boundaryidentity} consequently have positive curvature along $F$. Curvatures of planes containing normal vectors, general oblique planes, and planes in the regular part have not been controlled. This is exactly the limitation of the construction.
\end{proof}

\subsection{A first-order restriction without symmetry}
\begin{proposition}\label{prop:firstvariation}
Let $g_s$ be a smooth family of metrics on $S^2\times S^4$ with $g_0$ the round product metric. If the first variation of sectional curvature is nonnegative on every initially mixed plane, it vanishes on every such plane. In particular this vanishing is necessary if $\Sec_{g_s}>0$ for every sufficiently small $s>0$.
\end{proposition}
\begin{proof}
The product of a great circle in each sphere is a flat totally geodesic torus $T$ for $g_0$. Its second fundamental form for $g_s$ is $O(s)$, so its quadratic Gauss correction is $O(s^2)$. The first variation of its intrinsic Gaussian curvature therefore equals that of the ambient sectional curvature on $TT$. Differentiating Gauss--Bonnet gives
\[
 \int_T\left.\frac{d}{ds}\right|_{0}K_{g_s}(T_pT)\,dA_{g_0}=0,
\]
since the initial curvature is zero. A nonnegative integrand must vanish identically. Every mixed plane is tangent to a torus of this form. Positivity for $s>0$ implies nonnegativity of the right derivative at zero, proving the last assertion.
\end{proof}

This rules out a strictly positive first-order mixed-curvature direction, not a higher-order perturbation. The distinction is relevant to the recent perturbative construction reported in \cite{BH}, although its manifold and prescribed symmetry are different.

\subsection{The unresolved metric questions}
For both diagonal actions, the quotient and the underlying smooth manifold are now explicit. The arguments above exclude product and warped-product constructions, diagonal Cheeger deformations and their conformal classes, both totally geodesic product fibrations, and certain natural enlarged symmetry groups. The $\SO(3)$ action has the additional singular-orbit-radius obstruction of Corollary~\ref{cor:boundaryobstruction}.

None of these hypotheses is forced by the prescribed rank-one symmetry. The full metric space is \eqref{eq:allmetrics}, and Proposition~\ref{prop:boundarycompatible} shows that even the boundary restrictions have nonconstant compatible data. Thus the existence of an arbitrary positively curved invariant metric for either action is not decided by this paper. The smooth candidate classifications in Sections~\ref{sec:so3s3}--\ref{sec:su2s3} and the restricted metric obstructions in this section are logically separate conclusions.

\end{document}